\documentclass[10pt,twoside]{amsart}

\usepackage[
paper=a4paper,
text={135mm,197mm},
centering
]{geometry}
\usepackage{amssymb}
\usepackage[all]{xy}
\usepackage{pb-diagram,pb-xy}
\usepackage{graphicx,color,float}
\usepackage[bookmarks]{hyperref}
\usepackage{pinlabel}			

\usepackage[textwidth=1.5in]{todonotes}

\iftrue
\makeatletter
\def\@settitle{%
  \vspace*{-20pt}
  \begin{flushleft}%
    \baselineskip14\p@\relax
    \normalfont\bfseries\LARGE
    \@title
  \end{flushleft}%
}
\def\@setauthors{%
  \begingroup
  \def\thanks{\protect\thanks@warning}%
  \trivlist
  \large \@topsep30\p@\relax
  \advance\@topsep by -\baselineskip
  \item\relax
  \author@andify\authors
  \def\\{\protect\linebreak}%
  \authors
  \ifx\@empty\contribs
  \else
    ,\penalty-3 \space \@setcontribs
    \@closetoccontribs
  \fi
  \normalfont
  \@setaddresses
  \endtrivlist
  \endgroup
}
\def\@setaddresses{\par
  \nobreak \begingroup
  \small
  \raggedright
  \def\author##1{\nobreak\addvspace\smallskipamount}%
  \def\\{\unskip, \ignorespaces}%
  \interlinepenalty\@M
  \def\address##1##2{\begingroup
    \par\addvspace\bigskipamount\noindent
    \@ifnotempty{##1}{(\ignorespaces##1\unskip) }%
    {\ignorespaces##2}\par\endgroup}%
  \def\curraddr##1##2{\begingroup
    \@ifnotempty{##2}{\nobreak\noindent\curraddrname
      \@ifnotempty{##1}{, \ignorespaces##1\unskip}\/:\space
      ##2\par}\endgroup}%
  \def\email##1##2{\begingroup
    \@ifnotempty{##2}{\nobreak\noindent E-mail address%
      \@ifnotempty{##1}{, \ignorespaces##1\unskip}\/:\space
      \ttfamily##2\par}\endgroup}%
  \def\urladdr##1##2{\begingroup
    \def~{\char`\~}%
    \@ifnotempty{##2}{\nobreak\noindent\urladdrname
      \@ifnotempty{##1}{, \ignorespaces##1\unskip}\/:\space
      \ttfamily##2\par}\endgroup}%
  \addresses
  \endgroup
  \global\let\addresses=\@empty
}
\def\@setabstracta{%
    \ifvoid\abstractbox
  \else
    \skip@25\p@ \advance\skip@-\lastskip
    \advance\skip@-\baselineskip \vskip\skip@
    \box\abstractbox
    \prevdepth\z@ 
    \vskip-10pt
  \fi
}
\renewenvironment{abstract}{%
  \ifx\maketitle\relax
    \ClassWarning{\@classname}{Abstract should precede
      \protect\maketitle\space in AMS document classes; reported}%
  \fi
  \global\setbox\abstractbox=\vtop \bgroup
    \normalfont\small
    \list{}{\labelwidth\z@
      \leftmargin0pc \rightmargin\leftmargin
      \listparindent\normalparindent \itemindent\z@
      \parsep\z@ \@plus\p@
      
    }%
    \item[\hskip\labelsep\bfseries\abstractname.]%
}{%
  \endlist\egroup
  \ifx\@setabstract\relax \@setabstracta \fi
}
\def\section{\@startsection{section}{1}%
  \z@{-1.2\linespacing\@plus-.5\linespacing}{.8\linespacing}%
  {\normalfont\bfseries\Large}}
\def\subsection{\@startsection{subsection}{2}%
  \z@{-.8\linespacing\@plus-.3\linespacing}{.3\linespacing\@plus.2\linespacing}%
  {\normalfont\bfseries}}
\def\subsubsection{\@startsection{subsection}{3}%
  \z@{.7\linespacing\@plus.2\linespacing}{-1.5ex}%
  {\normalfont\itshape}}
\def\@secnumfont{\bfseries}
\makeatother
\fi 

\def\to{\mathchoice{\longrightarrow}{\rightarrow}{\rightarrow}{\rightarrow}}
\makeatletter
\newcommand{\shortxra}[2][]{\ext@arrow 0359\rightarrowfill@{#1}{#2}}
\def\longrightarrowfill@{\arrowfill@\relbar\relbar\longrightarrow}
\newcommand{\longxra}[2][]{\ext@arrow 0359\longrightarrowfill@{#1}{#2}}
\renewcommand{\xrightarrow}[2][]{\mathchoice{\longxra[#1]{#2}}%
  {\shortxra[#1]{#2}}{\shortxra[#1]{#2}}{\shortxra[#1]{#2}}}
\makeatother

\def\otimesover#1{\mathbin{\mathop{\otimes}_{#1}}}

\makeatletter
\def\Nopagebreak{\@nobreaktrue\nopagebreak}
\makeatother

\theoremstyle{plain}
\newtheorem{theorem}{Theorem}[section]
\newtheorem{proposition}[theorem]{Proposition}
\newtheorem{corollary}[theorem]{Corollary}
\newtheorem{lemma}[theorem]{Lemma}

\theoremstyle{definition}
\newtheorem{definition}[theorem]{Definition}

\newtheorem{example}[theorem]{Example}
\newtheorem{remark}[theorem]{Remark}

\def\Z{\mathbb{Z}}
\def\Q{\mathbb{Q}}

\def\C{\mathbb{C}}

\def\cP{\mathcal{P}}

\def\Hyp{M}

\def\Ker{\operatorname{Ker}}

\def\Im{\operatorname{Im}}
\def\Hom{\operatorname{Hom}}

\def\sign{\operatorname{sign}}

\def\rhot{\rho^{(2)}}
\def\bt{b^{(2)}}

\begin{document}

\title%
{Hidden torsion, 3-manifolds, and homology cobordism}

\author{Jae Choon Cha}

\address{Department of Mathematics\\
  POSTECH\\
  Pohang 790--784\\
  Republic of Korea\\
  and\linebreak
  School of Mathematics\\
  Korea Institute for Advanced Study \\
  Seoul 130--722\\
  Republic of Korea
}

\email{jccha@postech.ac.kr}

\author{Kent E. Orr}

\address{Department of Mathematics\\
  Indiana University, Bloomington\\
  Indiana 47405\\
  USA
}

\email{korr@indiana.edu}

\def\subjclassname{\textup{2010} Mathematics Subject Classification}
\expandafter\let\csname subjclassname@1991\endcsname=\subjclassname
\expandafter\let\csname subjclassname@2000\endcsname=\subjclassname
\subjclass{%
  57M27, 
  57N70
}

\keywords{Homology cobordism, 3-manifold, Local hidden torsion,
  $L^2$-signature}

\begin{abstract}
  This paper continues our exploration of homology cobordism of
  3-manifolds using our recent results on Cheeger-Gromov
  $\rho$-invariants associated to amenable representations.  We
  introduce a new type of torsion in 3-manifold groups we call hidden
  torsion, and an algebraic approximation we call local hidden
  torsion.  We construct infinitely many hyperbolic 3-manifolds which
  have local hidden torsion in the transfinite lower central subgroup.
  By realizing Cheeger-Gromov invariants over amenable groups, we show
  that our hyperbolic 3-manifolds are not pairwise homology cobordant,
  yet remain indistinguishable by any prior known homology cobordism
  invariants.  Additionally we give an answer to a question about
  transfinite lower central series of homology cobordant 3-manifold
  groups, asked by T. D. Cochran and M. H. Freedman.
\end{abstract}

\maketitle

\section{Introduction}

 
We investigate low dimensional manifolds via homology cobordism.  Recall that two closed 3-manifolds $M$ and $M'$ are (topologically) \emph{homology cobordant} if there is a 4-dimensional
topological cobordism $W$ between $M$ and $M'$ satisfying
$H_*(W,M)=0=H_*(W,M')$.  One can consider homology with twisted
coefficients as well.  We apply our recent results concerning the homology cobordism invariance of
Cheeger-Gromov $\rho$-invariants~\cite{Cha-Orr:2009-01} and utilizing
a new form of torsion in manifolds groups we call {\em local hidden torsion.}

Numerous central problems in low dimensional topology are special
cases of the problem of classifying homology cobordism types.
Concordant knots in a three manifold, $M,$ have homology cobordant
exteriors, with coefficients in the group ring
$\Z\pi_1M$~\cite{Cappell-Shaneson:1974-1}.  Hence, to classify knots
up to concordance we can understand homology cobordism of three
manifolds.  Casson and Freedman showed the knot slice problem contains
``universal'' four dimensional topological surgery problems whose
positive solution would yield a classification theory for topological
four manifolds~\cite{Casson-Freedman:1984-1,Freedman-Quinn:1990-1}.
Levine embedded the mapping class group, that is the group of isotopy
classes of surface diffeomorphisms, in the monoid of homology
cobordism classes of the mapping cylinders associated to surface
diffeomorphisms (see \cite{Levine:2001-1} for one boundary component
case, \cite{Cha-Friedl-Kim:2009-01} for the generalization to the
multi-boundary component case).  Hence, computing homology cobordism
classes of three manifolds potentially detects surface homeomorphisms
and the mapping class group.  (See
also~\cite{Goda-Sakasai:2008-1,Goda-Sakasai:2009-1,Cha-Friedl-Kim:2009-01}.)
And new results distinguishing smooth and topological knot concordance
give a laboratory for experiments in four dimensional smoothing
theory.

\subsection{Hidden torsion and main results}

Given a $4$-dimensional homology cobordism $W$ of $M$, the fundamental
groups of these manifolds relate in subtle and poorly understood ways.
In this paper, we exploit a new type of torsion we call {\em hidden
  torsion} in a manifold, and use that torsion, along with recent
results from~\cite{Cha-Orr:2009-01}, to compute Cheeger-Gromov
$\rho$-invariants as obstructions to homology cobordism.  Roughly
speaking, we say that a curve in a manifold $M$ represents
\emph{hidden torsion} if the curve has infinite order in $\pi_1(M)$
and is essential in any homology cobordism of $M$, but some power is
null-homotopic in some homology cobordism of~$M$
(Definition~\ref{definition:hidden-torsion}).  Note that a ``generic''
3-manifold has torsion-free fundamental group (e.g., all closed
irreducible non-spherical 3-manifolds, by the Geometrization
Conjecture).

We give an algebraic version of this notion using Vogel's homology
localization of groups (equivalently, Levine's algebraic closure of
groups) which we call {\em local hidden torsion.}  We discuss the Vogel
localization of groups in Section~\ref{subsection:localization}.

\theoremstyle{definition}
\newtheorem*{tempdef}{Definition~\ref{definition:local-hidden-torsion}}
\begin{tempdef}
  Let $G$ be a group, and $\widehat G$ the Vogel group localization
  of~$G$.  An element $g \in G$ is {\em local hidden torsion of $G$}
  if $g$ has infinite order in $G$, and nontrivial finite order in
  its image under $G \to \widehat G$.
\end{tempdef}
\noindent We remark that {\em local hidden torsion} and {\em hidden torsion}, as
defined above, agree for high dimensional manifolds. (See
Theorem~\ref{theorem:local-vs-geometric-hidden-torsion-high-dim}).

We exploit local hidden torsion in 3-manifold groups to realize
Cheeger-Gromov invariants, which first appeared as homology cobordism obstructions in~\cite{Cha-Orr:2009-01}.  We prove the following
theorem, and use this theorem to construct the examples in our main
result, Theorem~\ref{theorem:intro-main}.

\begin{theorem}
  \label{theorem:intro-hidden-local-torsion}
  There are closed hyperbolic 3-manifolds $M$ whose fundamental group
  has local hidden torsion lying in~$\pi_1(M)_\omega$.
\end{theorem}

\noindent Here $\{\pi_\alpha\}$ denotes the lower central series of a
group $\pi$, and $\omega$ is the first infinite ordinal.  Recall that
$\{\pi_\alpha\}$ is defined inductively by $\pi_1=\pi$,
$\pi_{\alpha}=[\pi,\pi_{\alpha-1}]$ for a successor ordinal $\alpha$,
and $\pi_{\alpha}=\bigcap_{\beta<\alpha}\pi_\beta$ for a limit
ordinal~$\alpha$.

The $3$-manifolds
we construct have non-trivial transfinite lower central series, a
phenomenon first observed and studied in~\cite{Cochran-Orr:1998-1}.

We use this local hidden torsion and a special case of our main
theorem in~\cite{Cha-Orr:2009-01} (see
Theorem~\ref{theorem:homology-cobordism-invariant-over-gamma-hat}) to
construct the examples satisfying parts (1) and (2) of the following
theorem.  This torsion hides within the intersection of the
lower central series and underlies our proof of parts (3) and (4) below.
 
\begin{theorem}
  \label{theorem:intro-main}
  There are infinitely many closed hyperbolic 3-manifolds $M=M_0,
  M_1,\ldots$ with the following properties:
  \begin{enumerate}
    \item For each $i$, there is a homology equivalence $f_i\colon M_i
    \to M$.  That is, $f_i$ induces an isomorphism on $H_*(-;\Z)$.
   \item Whenever $i\ne j$, $M_i$ and $M_i$ are not homology
    cobordant.
\end{enumerate}
Furthermore, all prior known homology cobordism obstructions fail to distinguish these examples.  In particular,
\begin{enumerate}
     \item[(3)] For any homomorphism $\phi\colon \pi_1(M)\to G$ with $G$
    torsion-free, the $L^2$-signature defects (= von
    Neumann-Cheeger-Gromov invariants) 
    $$
    \rhot(M,\phi) \mbox{ and }
    \rhot(M_i,\phi\circ {f_i}_*)
    $$ 
    are equal for each~$i$.  In
    particular Harvey's $\rho_n$-invariants \cite{Harvey:2006-1} of
    the $M_i$ are the same.
  \item[(4)] Similarly, the following homology cobordism invariants
    for the $M_i$ are equal:
   \begin{enumerate}
   \item Multi-signatures (= Casson-Gordon invariants) for prime power
     order characters in \cite{Gilmer:1981-1, Gilmer-Livingston:1983-1,
       Ruberman:1984-1, Cappell-Ruberman:1988-1}
   \item Atiyah-Patodi-Singer $\rho$-invariants for representations
     factoring through $p$-groups in
     \cite{Levine:1994-1,Friedl:2003-1}
   \item Twisted torsion invariants for representations factoring
     through $p$-groups in~\cite{Cha-Friedl:2010-01}
    \item Hirzebruch-type Witt-class-valued invariants from iterated
      $p$-covers in~\cite{Cha:2007-1}
    \end{enumerate}
  \end{enumerate}
\end{theorem}

Details and proofs are given as
Propositions~\ref{proposition:common-properties-of-examples},
\ref{proposition:torsion-free-L2-sign-does-not-detect},
\ref{proposition:p-group-invariants-do-not-detect}, and
Theorem~\ref{theorem:non-homology-cobordant-examples}.

We now outline the argument.  We begin with a surface bundle over a
circle with fundamental group $G = (\Z^t)^2\rtimes \Z$.  Here the
quotient group $\Z$ acts by negation on each factor of the free
abelian subgroup.  We compute the group localization of this group,
and show that this group localization has non-trivial transfinite
lower central series, and that the $\omega$ term in this series
contains torsion.  We construct an homology cobordism from our surface
bundle to a new three manifold with the property that the image of the
fundamental group of the new three manifold in its group localization
contains some of the above torsion.  Like the group localization, this
new three manifold has non-trivial transfinite lower central series.
We alter this three manifold, preserving it's homology type by
replacing a neighborhood of a curve representing hidden torsion with a
new homology circle, the complement of a knot, chosen to alter an
appropriately chosen Cheeger-Gromov $\rho$-invariant of the three
manifold associated to this torsion group element.

The resulting homology cobordism classes of the $M_i$ are
indistinguishable via any prior known invariant, as stated in (3)
and~(4).  This follows, with some work, from observing that the
relevant local hidden torsion vanishes in the group nilpotent
completion, and threrefore cannot be detected in any nilpotent
quotient group, and in particular, in any $p$-group.  Additionally,
the effect of this construction along the local hidden torsion is
invisible via any $L^2$-signature associated to a representation to a torsion-free group.

By contrast, the Cheeger-Gromov invariants employed to detect these examples arise from representations to 
infinite, non-nilpotent amenable groups
with torsion, in which the local hidden torsion is not eliminated.

Additional applications, by the first author, of the main results
of~\cite{Cha-Orr:2009-01} exploiting torsion can be found
in~\cite{Cha:2010-01}.

  We obtain one additional result from our construction and answer a
  question posed by Cochran and Freedman (see Kirby's problem list
  \cite[Problem~3.78]{Kirby:problem-list-1995-edition}): \emph{Is the
    lower central series length of a 3-manifold group invariant under
    homology cobordism?}

  Recall that the lower central series length of a group $\pi$ is
  defined to be the smallest ordinal $\alpha$ satisfying
  $\pi_\alpha=\pi_{\alpha+1}$.  Lower central series quotients $\pi/\pi_q$, $q$ finite, are preserved
  under homology cobordism by Stallings's
  Theorem~\cite{Stallings:1965-1}.  Therefore, one must consider 3-manifolds with transfinite lower central series length.  In \cite[Proposition~9.1]{Cochran-Orr:1998-1}, Cochran and
  the second author considered an algebraic analogue formulated for
  finitely presented groups and gave a negative answer.  We give a
  3-manifold group counterexample.

  \begin{theorem}
    \label{theorem:answer-to-cochran-freedman}
    There are two homology cobordant closed 3-manifolds with fundamental groups of differing lengths, one having length $\omega$,  and the other length $>\omega$.
  \end{theorem}

\subsubsection*{Convention}

In this paper all manifolds are assumed to
be oriented and compact.

\subsubsection*{Acknowledgments}
The authors thank the anonymous referee for detailed comments.
The first author was supported by National Research
Foundation of Korea (NRF) grants funded by the Ministry of Education,
Science and Technology (No. 2010--0029638 and 2010--0011629).  The
second author was supported by NSF grant DMS--0707078, and Simons Foundation grant 209082.

\section{Homology localization and hidden torsion}
\label{section:homology-localization-and-hidden-torsion}

\subsection{Homology localization of groups}
\label{subsection:localization}

First introduced by Bousfield~\cite{Bousfield:1975-1}, Vogel modified
the construction of homology localization, and with keen insight,
revealed its fundamental role in the study of manifold
embeddings~\cite{Vogel:1978-1}.  For groups, Levine's theory of
algebraic closures provides a very influential algebraic approach to
the homology localization~\cite{Levine:1989-2,Levine:1989-1}.  (For
related works, see, e.g., \cite{LeDimet:1988-1, Vogel:1978-1,
  Cha:2004-1, Cochran-Harvey:2004-1, Sakasai:2006-1, Cha:2007-1,
  Cha:2007-2, Cha-Orr:2009-01, Cochran-Harvey:2008-1, Gutierrez:1979-1, Heck:2010-01}.)

The homology localization we use in this paper is a specific case of
Vogel's more complicated general theory in \cite{Vogel:1978-1}, and follows the combinatorial approach introduced in~\cite{Levine:1989-1,Levine:1989-2,Levine:1992-1}, suitably modified to remove the normal closure condition used therein.  Explicit definitions that precisely fit our
purpose first appear in \cite{Cha:2004-1} and \cite{Cha-Orr:2009-01}.
For the reader's convenience, we recall these:

\begin{definition}
  \label{definition:homology-localization}
  Let $R$ be a commutative ring with unity.
  \begin{enumerate}
  \item A group homomorphism $\alpha\colon \pi \to G$ is called
    \emph{$R$-homology 2-connected}
    if $f$ induces an
    isomorphism on $H_1(-;R)$ and a surjection on $H_2(-;R)$.  We
    denote by $\Omega^R$ the collection of $\alpha\colon \pi \to G$
    with $\pi$ and $G$ finitely presented and $\alpha$ $R$-homology
    2-connected. We simply say $\alpha$ is 2-connected when
    $R=\Z$.
  \item A group $K$ is \emph{$R$-local} if for any $\alpha\colon\pi\to
    G$ in $\Omega^R$ and for any homomorphism $f\colon \pi\to K$,
    there is a unique homomorphism $g\colon G\to K$ satisfying $g\circ
    \alpha=f$.
  \item The \emph{Vogel-Levine $R$-homology localization} of a group
    $G$ is a group $\widehat G$ endowed with a homomorphism $p_G\colon
    G\to \widehat G$ such that $\widehat G$ is $R$-local and $p_G$ is
    universal (initial) among morphisms with this property. That is,
    for any $f\colon G\to K$ with $K$ local, there is a unique
    $g\colon \widehat G \to K$ satisfying $g\circ p_G = f$.
  \end{enumerate}
\end{definition}

One knows that for any group $G$ there exists unique $(\widehat G,
p_G)$ satisfying the above.  (For a proof, e.g.,
see~\cite{Cha:2004-1}.)  Useful consequences of
Definition~\ref{definition:homology-localization} are the following
properties:
\begin{enumerate}
\item The association $G\mapsto \widehat G$ is a functor. 
\item $\{p_G\}$ is a natural transformation.
\item Any $\alpha\colon \pi\to G$ in $\Omega^R$
gives rise to an isomorphism $\widehat\alpha\colon \widehat\pi
\to\widehat G$.
\end{enumerate}
In this paper homology localization always refers to the group localization in
Definition~\ref{definition:homology-localization}.

A homology equivalence induces a two-connected homomorphism between
fundamental groups.  Hence, and central to applications to manifolds,
a homology cobordism $W$ between manifolds $M_i, i = 0,1$, determines
inclusion-induced isomorphisms
$\widehat{\pi_1(M_i)} \cong \widehat{\pi_1(W)}$.

\subsection{Hidden torsion}
Local hidden torsion is an often more accessible homological
approximation to hidden torsion, suitable for some applications,
including the primary application of this paper.  We prove in
Theorem~\ref{theorem:local-vs-geometric-hidden-torsion-high-dim} below
that hidden torsion and local hidden torsion agree in high dimensional
manifolds.  This intrinsically low dimensional distinction deserves
more extensive examination in a future paper.

We recall the definitions from the introduction:

\begin{definition}
 \label{definition:local-hidden-torsion}
 Let $G$ be a group, and $\widehat G$ the homology localization
 of~$G$.  An element $g \in G$ is {\em local hidden torsion of $G$}
 if $g$ has infinite order in $G$, and nontrivial finite order in its
 image under $G \to \widehat G$.
\end{definition}

\begin{definition}
  \label{definition:hidden-torsion}
  For a manifold $M$, an element $g\in \pi_1(M)$ is called
  \emph{hidden torsion of $M$} if $g$ has infinite order in $\pi_1(M)$
  and is essential in any homology cobordism $W$ of $M$, but for some
  homology cobordism $W$ of $M$, the image of $g$ in $\pi_1(W)$ has
  finite order.
\end{definition}

We remark that if $g\in \pi_1(M)$ is hidden torsion, then for any
$N$ homology cobordant to $M$, there is a homology cobordism $W$
between $M$ and $N$ satisfying the following: the image of $g$ in
$\pi_1(W)$ has finite order.  This follows from the obvious fact that
if $V$ is a homology cobordism with $\partial V=M\cup -N$ and $V'$ is
a homology cobordism with $\partial V'=M\cup -M'$ via which
$g\in\pi_1(M)$ is hidden torsion, then $W=V' \cup_{M'} -V' \cup_{M} V$
is a homology cobordism from $M$ to~$N$.

The following hidden torsion example in a hyperbolic 3-manifold may be
viewed as a geometric analogy to the fact that an element $g$ in a
group $G$ may have infinite order even when its homology class $[g]\in
H_1(G)$ has nontrivial finite order.

\begin{example}
  Suppose $K$ is a hyperbolic slice knot in $S^3$.  Let $M$ be the
  $(n/1)$-surgery along $K$ on~$S^3$, so that $H_1(M)=\Z/n$.  By
  Thurston's hyperbolic Dehn surgery theorem~\cite{Thurston:1978-1},
  $M$ is hyperbolic for all but finitely many $n$, and consequently
  $\pi_1(M)$ is torsion-free for those~$n$.  Since the meridian $\mu$
  of $K$ in $M$ is a generator of $H_1(M)$, it follows that the order
  of $\mu$ in $\pi_1(W)$ is at least $n$ for any homology cobordism
  $W$ of $M$.  In particular, $\mu$ is not null-homotopic in~$W$.
  
  Taking a concordance $C\cong S^1\times[0,1]$ in $S^3\times[0,1]$
  between $K$ and the unknot and by filling in the exterior of $C$
  with $D^1\times S^1\times[0,1]$ along the $(n/1)$-framing, we obtain
  a homology cobordism, say $W$, between $M$ and the lens space
  $L=L(n,1)$.  Since $\mu$ is isotopic to the generator of
  $\pi_1(L)=\Z/n$ in $W$, $\mu$ has order $n$ in~$\pi_1(W)$.  This
  shows that $\mu$ represents hidden torsion of~$M$.

  Since $H_2(\Z/n)=0$, the abelianization map $\pi_1(M)\to
  H_1(M)=\Z/n$ induces an isomorphism on homology localization.  Since
  the homology localization of any abelian group is the group itself,
  it follows that the localization of $\pi_1(M)$ is the abelianization
  map $\pi_1(M)\to \Z/n$.  Therefore $\mu$ represents local hidden
  torsion of~$\pi_1(M)$.
\end{example}


\begin{theorem}
  \label{theorem:local-vs-geometric-hidden-torsion-high-dim}
  Suppose $M$ is a closed $n$-manifold with $n>3$.  Then an element
  $g\in \pi_1(M)$ is hidden torsion of $M$ if and only if $g$ is local
  hidden torsion of~$\pi_1(M)$.
\end{theorem}

\begin{proof}
  For convenience we write $\pi=\pi_1(M)$.  It is known (e.g., see
  \cite[Theorem~2.6]{Cha:2004-1}) that the homology localization
  $\widehat\pi$ is the direct limit of a sequence of 2-connected
  homomorphisms between finitely presented groups~$G_i$:
  \[
  \pi=G_0 \to G_1 \to \cdots \to G_n \to \cdots
  \]

  First we prove the only if direction.  Suppose $g\in\pi$ is hidden
  torsion.  If $g$ were trivial in $\widehat\pi$, then (the image of)
  $g$ would be trivial in some $G_i$.  But, since $n>3$, by appealing
  Lemma~\ref{lemma:homology-cobordism-with-given-fundamental-group}
  below, there is a homology cobordism $W$ of $M$ for which
  $\pi_1(W)=G_i$ and the inclusion-induced map $\pi_1(M)\to \pi_1(W)$
  is equal to the above $\pi\to G_i$.  This contradicts the hypothesis
  that $g$ is hidden torsion.  It follows that $g$ is nontrivial
  in~$\widehat\pi$.  Also, there is a homology cobordism $W$ for which
  $g^k$ is trivial in $\pi_1(W)$ for some $k>0$, by the definition of
  hidden torsion.  It follows that $g^k$ is trivial in
  $\widehat\pi_1(W) \cong \widehat\pi$.

  For the if part, suppose that $g$ is local hidden torsion.  For any
  homology cobordism $W$ of $M$, $\widehat{\pi_1(W)}\cong
  \widehat\pi$.  Therefore $g$ is nontrivial in $\widehat{\pi_1(W)}$,
  and consequently in~$\pi_1(W)$.  Since $g$ has finite order in
  $\widehat\pi$, so does $g$ in some~$G_i$.  Appealing to
  Lemma~\ref{lemma:homology-cobordism-with-given-fundamental-group}
  below, choose a homology cobordism $W$ of $M$ satisfying
  $\pi_1(W)\cong G_i$.  Then $g$ has finite order in~$\pi_1(W)$.  This
  shows that $g$ is local torsion.
\end{proof}

The following fact used in the above proof of
Theorem~\ref{theorem:local-vs-geometric-hidden-torsion-high-dim} is
not due to us and seems known to experts.  Whereas the key idea is
used in various applications in the literature, we did not find an
explicitly written proof elsewhere.  We give a proof for the
convenience of the readers.

\begin{lemma}
  \label{lemma:homology-cobordism-with-given-fundamental-group}
  Suppose $M$ is a closed $n$-manifold with $n>3$ and $\phi\colon
  \pi_1(M) \to G$ is a 2-connected homomorphism with $G$ finitely
  presented.  Then there is a homology cobordism $W$ of $M$ such that
  $\pi_1(W)=G$ and the inclusion $M\to W$ induces~$\phi$.
\end{lemma}

\begin{proof} Choose generators $g_1,\ldots,g_k$ of~$G$.  For
  convenience we denote by $F\langle g_j\rangle$ the free group on the
  $g_1,\ldots,g_k$ viewing these as symbols, and often regard an
  element in $F\langle g_j\rangle$ as its image in $G$ under the
  obvious surjection $F\langle g_j\rangle \to G$.

  Since $\phi$ gives rise to an isomorphism on $H_1(-)$, for each
  $i=1,\ldots,k$ there exist $\gamma_i\in \pi_1(M)$ and a product of
  commutators $h_i \in F\langle g_j\rangle$ satisfying $g_i =
  \phi(\gamma_i)\cdot h_i$ in~$G$.  Now consider the $(n+1)$-manifold
  $V$ obtained by attaching handles to $M\times[0,1]$ as follows:
  \[
  V = M\times[0,1] \cup (\text{$k$ 1-handles}) \cup (\text{$k$
    2-handles})
  \]
  where the $j$-th 1-handle corresponds to the $j$-th generator $g_j$
  so that the fundamental group of $V^{(1)}=M\times[0,1]\cup
  (\text{1-handles})$ is identified with free product
  $\pi_1(M)*F\langle g_j\rangle$, and the $i$-th 2-handle is attached
  along an embedded circle representing $g_i^{-1}\gamma_i h_i
  \in\pi_1(V^{(1)})$.  Obviously there is a homomorphism $\psi\colon
  \pi_1(V) \to G$ that the given $\phi\colon\pi_1(M)\to G$ factors
  through.
  Also, computing $H_*(V,M)$ from the handle decomposition, it is
  easily seen that the $V$ is a homology cobordism, since each $h_i$
  is in the commutator subgroup.

  Let $N=\Ker \psi$.  Note that $N$ is contained in the commutator
  subgroup $[\pi_1(V),\pi_1(V)]$ since $\psi$ induces an isomorphism
  on $H_1(-)$.  Since $\psi$ is surjective and $G$ is finitely
  presented, $N$ is normally finitely generated in $\pi_1(V)$, namely
  for some finitely many elements $n_1,\ldots, n_r\in N$, the normal
  subgroup in $\pi_1(V)$ generated by the $n_i$ is equal
  to~$\pi_1(V)$.  Choosing disjoint embedded circles in the interior
  of $V$ representing the $n_i$ and performing surgery on $V$ along
  these circles, we obtain an $(n+1)$-manifold, say~$V'$.  By standard
  arguments, using the fact that each $n_i$ lies in the commutator
  subgroup of $\pi_1(V)$, one shows that $\pi_1(V')=G$,
  $H_2(V')=H_2(V)\oplus \Z^r$, and $H_i(V')=H_i(V)$ for $i>2$, where
  the $\Z^r$ factor of $H_2$ is introduced by surgery.

  We will do surgery to eliminate the $\Z^r$ factor of $H_2(V')$.  For
  this purpose, first we obtain appropriate spherical elements
  in~$H_2$ as follows.
  Since $\phi\colon\pi_1(M)\to G$ is 2-connected, the map
  \[
  H_2(V)\cong H_2(M)\to H_2(G)\cong H_2(\pi_1(V'))
  \]
  is surjective.  It follows that we can choose elements
  $c_i=(x_i,y_i) \in H_2(V')=H_2(V)\oplus \Z^r=H_2(M)\oplus \Z^r$
  ($i=1,\ldots,r$) satisfying the following: the $y_i$ form a basis of
  $\Z^r$ and the $c_i$ lie in the kernel of the map $H_2(V')\to
  H_2(\pi_1(V'))$.  By the exact sequence
  \[
  \pi_2(V') \to H_2(V') \to H_2(\pi_1(V')) \to 0
  \]
  it follows that the elements $c_i$ are spherical.  We may assume
  that the $c_i$ are given as disjoint embedded 2-spheres, and can
  perform surgery along the $c_i$ on $V'$, since this is below middle
  dimension ($n+1\ge 5$).  Its effect is exactly eliminating the
  $c_i$; the result of surgery, say $W$, is an $(n+1)$-manifold with
  $H_*(W)\cong H_*(V) \cong H_*(M)$ as desired.
\end{proof}

In the remainder of this paper, we construct three manifolds whose groups contain local hidden torsion that cannot be seen in the nilpotent completion,
and then use this hidden torsion to construct new manifolds homology equivalent, but not homology cobordant to each other.

\section{Torus bundle group and homology cobordism invariance}
\label{section:computation-for-torus-bundle-group}

We construct homology cobordism invariants in three steps.  In
Section~\ref{subsection:homology-localization-of-Gamma}, we compute
the integral homology localization $\widehat\Gamma$ of a torus bundle
group~$\Gamma$.  In
Section~\ref{subsection:Mixed-coefficient-commutator-series}, we
construct a representation of $\widehat\Gamma$ into an amenable
$D(\Z/p)$-group, where $D(\Z/p)$ denotes Strebel's class
in~\cite{Strebel:1974-1}. Lemma~6.8 in~\cite{Cha-Orr:2009-01}
clarifies how to achieve this using the ``mixed type commutator
series'' defined in~\cite{Cha:2010-01}.  (Here successive terms of the
derived series of a group arise from distinct coefficient subrings of
the rational numbers.)  In
Section~\ref{subsection:homology-cobordism-invariant}, by applying the
$\Z/p$-coefficient version of a key result in \cite{Cha-Orr:2009-01}
(stated as Theorem~\ref{theorem:cha-orr-L2-signature} in this
section), we obtain our homology cobordism invariant.

We remark that we need the localization of $\Z$ at $p$, $\Z_{(p)}=$,
and $\Z/p$, as coefficients in the latter two steps to apply the main
result of \cite{Cha-Orr:2009-01}. While it suffices to use integral
group homology localization in the first step, we present our
computation of group localization for any subring $R$ of $\Z_{(2)}$.
This small generalization may prove useful elsewhere, and omitting it
does not simplify the computation in any way.  (See
Theorem~\ref{theorem:computation-of-localization}.)

\subsection{Homology localization of a twisted torus bundle group}
\label{subsection:homology-localization-of-Gamma}

We consider the group $\Gamma$ defined below, and compute its group
localization.  As stated previously, this group localization has a
non-trivial transfinite lower central series which contains torsion.
We remark that Levine investigated earlier a different version of
algebraic closure for this group~\cite[Proposition~2]{Levine:1991-1}.

Let $\Gamma$ be the group
$$
\Gamma = (\Z^t)^2 \rtimes \Z = \langle x,y,t \mid [x,y]=1,\ txt^{-1}=x^{-1},\ tyt^{-1}=y^{-1} \rangle.
$$
In the first expression the $\Z$ factor is the subgroup generated by
$t$, and $\Z^t$ is the $\Z[t^{\pm1}]$-module with underlying abelian
group $\Z$ on which $t$ acts by negation $t\cdot a = -a$.  The
elements $x$ and $y$ generate the two $\Z^t$ factors.  This group
$\Gamma$ is the fundamental group of an oriented torus bundle over $S^1$ whose
monodromy on $S^1\times S^1$ is given by $(z,w)\mapsto
(z^{-1},w^{-1})$.

We construct the homology localization of our group $\Gamma$ as the limit of groups $\Gamma_{2k-1}$, $k$ positive.
\[
\Gamma_{2k-1} = \Big\langle x,y,t\ \Big|
\  [x,y]^{(2k-1)^2},\ [[x,y],t],\ [[x,y],x],\ [[x,y],y],\
  txt^{-1}=x^{-1},\ tyt^{-1}=y^{-1}
\Big\rangle
\]
When $k=1$, we have $\Gamma_1=\Gamma$.  The group $\Gamma_{2k-1}$ is a
central extension with center the order $(2k-1)^2$ subgroup generated by~$[x,y]$ and with quotient group isomorphic to~$\Gamma$.  For later convenience, we denote this subgroup and quotient group by
$$
\big(\tfrac{1}{(2k-1)^2}\Z\big)/\Z \qquad \mbox{and} \qquad
\big(\tfrac{1}{2k-1}\Z^t\big)^2\times \Z,$$
respectively.  That is, $\Gamma_{2k-1}$ is the extension:
\[
1\to \big(\tfrac{1}{(2k-1)^2}\Z\big)/\Z \to \Gamma_{2k-1} \to
\big(\tfrac{1}{2k-1}\Z^t\big)^2 \rtimes \Z \to 1
\]

Define a function $\phi_{2k-1}\colon \Gamma \to \Gamma_{2k+1}$ by $t\mapsto
t$, $x \mapsto x^{2k+1}$, $y \mapsto y^{2k+1}$.  One easily sees
that this assignment gives us a well-defined homomorphism: by the
identities
\[
  [ab,c]=a[b,c]a^{-1}[a,c],\quad
  [c,ab]=[c,a]a[c,b]a^{-1}
\]
we have 
$$
[x,y] \mapsto [x^{2k+1},y^{2k+1}]=[x,y]^{(2k+1)^2},$$
and using this, the relations are verified.

Similarly, for $2k-1 \mid 2\ell-1$, we define a homomorphism $
\Gamma_{2k-1} \to \Gamma_{2\ell-1} $ given by 
\[
t\mapsto t, \quad x\mapsto
x^{(2\ell-1)/(2k-1)}, \quad  y\mapsto y^{(2\ell-1)/(2k-1)}.
\]
Then the groups $\Gamma_{2k-1}$ together with these homomorphisms form a direct
system.

Recall that the (classical) localization of $\Z$ at the prime~2 is
denoted by $\Z_{(2)}=\{a/b\in \Q \mid b$ is odd$\}$.

\begin{theorem}
  \label{theorem:computation-of-localization}
  Suppose $R$ is a subring of~$\Z_{(2)}$.  Then for any $k$, the
  $R$-homology localization of $\Gamma_{2k-1}$ is the colimit
  $\varinjlim_{k} \Gamma_{2k-1}$ endowed with $\Gamma_{2k-1} \to
  \varinjlim_{k} \Gamma_{2k-1}$.
\end{theorem}

In particular Theorem~\ref{theorem:computation-of-localization} gives
that the homology localization of $\Gamma=\Gamma_1$ is $\varinjlim_k
\Gamma_{2k-1}$.

By Theorem~\ref{theorem:computation-of-localization},
we have the following commutative diagram, where vertical maps are the
limit maps, and $\Z_{(2)}/\Z$ is a central subgroup
of~$\widehat\Gamma$:
\[
\xymatrix{
  1 \ar[r]
  & \big(\tfrac{1}{(2k-1)^2}\Z\big)/\Z \ar[r]\ar@{^{(}->}[d]
  & \Gamma_{2k-1} \ar[r]\ar@{^{(}->}[d]
  & \big(\tfrac{1}{2k-1}\Z^t\big)^2 \rtimes \Z \ar[r]\ar[r]\ar@{^{(}->}[d]
  & 1
  \\
  1 \ar[r]
  & \Z_{(2)}/\Z \ar[r]
  & \widehat\Gamma \ar[r]
  & (\Z_{(2)}^t)^2 \rtimes \Z \ar[r]
  & 1
}
\]

\begin{proof}[Proof of Theorem~\ref{theorem:computation-of-localization}]
    By abelianizing the above presentation, we see that
    $H_1(\Gamma_{2k-1})\cong (\Z/2)^2 \times \Z$ generated by
    $x$,~$y$,~$t$.  Since we need it later, we give a proof of the
    following stronger fact: 
    the commutator subgroup $[\Gamma_{2k-1},\Gamma_{2k-1}]$ is generated by
    $x^2$,~$y^2$.  In fact, the relations $txt^{-1}=x^{-1}$,
    $tyt^{-1}=y^{-1}$ implies $[t,x]=x^{-2}$ and $[t,y]=y^{-2}$.  The
    commutator $[x,y]$ is also contained in the subgroup generated by
    $x^2$, $y^2$ since
    $[x,y][x^2,y^2]^{k^2-k}=[x,y][x,y]^{4k^2-4k}=[x,y]^{(2k-1)^2}=1$.
    In addition the subgroup generated by $x^2$ and $y^2$ is a normal
    subgroup since $tx^2t^{-1}=x^{-2}$ and $yx^2y^{-1} =
    [y,x^2]x^2=[y,x]^2x^2=[x,y]^{-2}x^2$.  This verifies the claim.
  It follows that $\phi_{2k-1}\colon \Gamma\to\Gamma_{2k-1}$ induces
  an isomorphism on $H_1(-)$.

  Also, it is verified that $H_2(\Gamma_{2k-1})\cong\Z$ and
  $\phi_{2k-1}$ induces an isomorphism on $H_2(-)$, by computation
  using the Lyndon-Hochschild-Serre spectral sequence twice, first for
  the central extension
  \[
  1 \to \big(\tfrac{1}{(2k-1)^2}\Z\big)/\Z \to A_{2k-1} \to
  \big(\tfrac{1}{2k-1}\Z^t\big)^2 \to 1
  \]
  where $A_{2k-1}$ is the subgroup of $\Gamma_{2k-1}$ generated by $x$
  and $y$, and then for
  \[
  1 \to A_{2k-1} \to \Gamma_{2k-1} \to \Z \to 1.
  \]
   More details are as follows.  First observe that the relation
   $[[x,y],t]$ of $\Gamma_{2k-1}$ is redundant, since
   $txt^{-1}=x^{-1}$ and $tyt^{-1}=y^{-1}$ imply
   $t[x,y]t^{-1}=[x^{-1},y^{-1}]$ and since we obtain
   $[x^a,y^b]=[x,y]^{ab}$ for any integers $a$,~$b$ by combining the
   relations $[[x,y],x]=1=[[x,y],y]$ with the identities for $[ab,c]$
   and $[c,ab]$ as above.  Therefore $\Gamma_{2k-1}$ is an HNN
   extension of the subgroup
   \[
   A_{2k-1} = \big\langle x,y \ \big|\  [x,y]^{(2k-1)^2},
   [[x,y],x],[[x,y],y] \big\rangle
   \]
   by the infinite cyclic group generated by $t$ with the action
   $txt^{-1}=x^{-1}$, $tyt^{-1}=y^{-1}$.  This is the second extension
   given above.  Also, the first extension above follows immediately
   from our presentation of~$A_{2k-1}$.

   The first extension is central and has spectral sequence with
   $E^2_{p,q}\cong H_p((\tfrac{1}{2k-1}\Z^t)^2)\otimes
   H_q(\frac{1}{(2k-1)^2}\Z/\Z)$.  We have $E_{0,2}^\infty=0$,
   $E_{1,1}^\infty=E_{1,1}^2=(\tfrac{1}{(2k-1)^2}\Z^t/\Z)^2$, and
   $E_{2,0}^\infty$ is the kernel of
   \[
   d_{2,0}\colon H_2\big((\tfrac{1}{2k-1}\Z^t)^2\big) \to
   H_1\big(\tfrac{1}{(2k-1)^2}\Z/\Z\big).
   \]
   Viewing $d_{2,0}$ as the transgression map, one sees that $d_{2,0}$
   is identical with the projection $\frac{1}{(2k-1)^2}\Z \to
   \frac{1}{(2k-1)^2}\Z/\Z$.  It follows that $H_2(A_{2k-1})$ is given
   by
   \[
   0 \to (\tfrac{1}{(2k-1)^2}\Z^t/\Z)^2 \to H_2(A_{2k-1}) \to \Z \to 0.
   \]
   It is easily seen that $H_1(A_{2k-1})=(\frac{1}{2k-1}\Z^t)^2$ from the
   presentation.

   The expression of $\Gamma_{2k-1}$ as an extension of $A_{2k-1}$ by
   $\Z$ gives a spectral sequence with
   $E_{p,q}^2=H_p(\Z;H_q(A_{2k-1}))$.  From this we obtain the Gysin sequence
   \[
   H_2(A_{2k-1})\xrightarrow{t_*-1} H_2(A_{2k-1}) \to H_2(\Gamma_{2k-1})
   \to H_1(A_{2k-1}) \xrightarrow{t_*-1} H_1(A_{2k-1}).
   \]
   Combining this with our computation of $H_1$, $H_2$ of $A_{2k-1}$
   it follows that $H_2(\Gamma_{2k-1})=\Z$ as claimed.

   Now, from the above
 it follows that $\phi_{2k-1}\colon\Gamma\to \Gamma_{2k-1}$, and
 consequently $\Gamma_{2k-1}\to\Gamma_{2\ell-1}$ for $2k-1\mid
 2\ell-1$, are in~$\Omega^R$ for any~$R$.  Therefore, the homomorphism
 $\Gamma \to \Gamma_{2k-1}$ induces an isomorphism on localization,
 and it suffices to show that $\varinjlim_k \Gamma_{2k-1}$ is the
 $R$-homology localization of~$\Gamma$.
  
  For any $R$-local group $K$ and a homomorphism $f\colon \Gamma\to
  K$, there is a unique homomorphism $g_{2k-1} \colon \Gamma_{2k-1}
  \to K$ such that $f=g\phi_{2k-1}$ since $\phi_{2k-1} \in \Omega^R$.
  It follows that the limit $g=\varinjlim_{k} g_{2k-1}$ is the unique
  homomorphism making the following diagram commute:
  \[
  \begin{diagram}
    \node{\Gamma}\arrow{e}\arrow{s,l}{f}
    \node{\textstyle\varinjlim_k\Gamma_{2k-1}}\arrow{sw,b}{g}
    \\
    \node{K}
  \end{diagram}
  \]
  
  Now, by the following proposition, $\varinjlim_k \Gamma_{2k-1}$ is
  the $R$-homology localization of~$\Gamma$.
\end{proof}

\begin{proposition}
  \label{proposition:hat-Gamma-is-local}
  If $R$ is a subring of $\Z_{(2)}$, then $\varinjlim_k\Gamma_{2k-1}$
  is $R$-local.
\end{proposition}

This follows from the fact that $\Z_{(2)}^t$ is a local module over
$\Z[t^{\pm1}]$ in the sense of Cohn and Vogel.  For completeness, we
present a proof of Proposition~\ref{proposition:hat-Gamma-is-local} in
Appendix~\ref{section:local-groups-and-local-modules}.

The following shows that  $\widehat\Gamma$, and for $k > 1$, $\Gamma_{2k-1}$ have
nontrivial $\omega$-term in the lower central series.



\begin{lemma}
  \label{lemma:omega-lower-central-series-gamma-hat}
  The transfinite lower central subgroup $(\Gamma_{2k-1})_\omega$ is
  equal to the subgroup $(\frac{1}{(2k-1)^2}\Z)/\Z$ of
  $\Gamma_{2k-1}$, and $(\Gamma_{2k-1})_{\omega+1}$ is trivial.
  Consequently $\widehat\Gamma_\omega = \Z_{(2)}/\Z$ and
  $\widehat\Gamma_{\omega+1}$ is trivial.
\end{lemma}

\begin{proof}
  As previously shown, $(\Gamma_{2k-1})_2 = [\Gamma_{2k-1},\Gamma_{2k-1}]$
  is the subgroup generated by $x^2$,~$y^2$.  By an induction using
  the same argument, $(\Gamma_{2k-1})_{q+1}$ is the subgroup generated by
  $x^{2^q}$,~$y^{2^q}$.  Therefore $(\Gamma_{2k-1})_\omega$
  lies in the kernel of $\Gamma_{2k-1} \to
  (\frac{1}{2k-1}\Z^t)^2\rtimes\Z$, which is the subgroup
  $(\frac{1}{(2k-1)^2}\Z)/\Z$.

  We claim that the generator $[x,y]$ of $(\frac{1}{(2k-1)^2}\Z/\Z)$
  lies in $(\Gamma_{2k-1})_\omega$.  
  %
  %
  Since $x^{2^q}\in (\Gamma_{2k-1})_{q+1}$, it follows that 
  $[x^{2^q},y]=[x,y]^{2^q}$ lies in $(\Gamma_{2k-1})_{q+2}$.  Since $[x,y]$
  has odd order, it follows that $[x,y]\in (\Gamma_{2k-1})_{q+2}$.
  Since this holds for any $q$, $[x,y] \in (\Gamma_{2k-1})_\omega$ as
  claimed.  It follows that $(\Gamma_{2k-1})_\omega =
  (\frac{1}{(2k-1)^2}\Z)/\Z$.

  Since the generator $[x,y]$ of $(\Gamma_{2k-1})_\omega$ is central,
  $(\Gamma_{2k-1})_{\omega+1}$ is trivial.

  The conclusion on $\widehat\Gamma$ is obtained by taking the direct
  limit.
\end{proof}

In the remaining part of this paper, we will always use the integral
homology localization of $\Gamma$, which we denote
by~$\widehat\Gamma$.

\subsection{Mixed-coefficient commutator series}
\label{subsection:Mixed-coefficient-commutator-series}

The homology localization $\widehat\Gamma$ is not in Strebel's class
$D(R)$ for $R=\Z$, $\Q$ or $\Z/p$, while we need (amenable and)
$D(R)$-groups to apply the main result of \cite{Cha-Orr:2009-01}.  In
this subsection, we construct a representation of $\widehat\Gamma$ to
an amenable $D(\Z/p)$-group, by separating $p$-torsion elements from
those with order coprime to~$p$.

For the above purpose, we consider the mixed-type coefficient
commutator series $\{\cP^n G\}$ of a group $G$ which was introduced
in~\cite{Cha:2010-01}.  Suppose $\cP=(R_0,R_1,\ldots)$ be a sequence
of commutative rings with unity.  Then for a group $G$, the series
$\{\cP^n G\}$ is defined inductively by
\begin{align*}
  \cP^0 G &= G,
  \\
  \cP^{n+1}G &= \smash[t]{
    \Ker\Big\{\cP^{n}G \to
    \frac{\cP^{n}G}{[\cP^{n}G,\cP^{n}G]} \to
    \frac{\cP^{n}G}{[\cP^{n}G,\cP^{n}G]} \otimesover{\Z} R_n \Big\}.
  }
\end{align*}

The subgroup $\cP^n G$ is a characteristic normal subgroup of~$G$.
 
We use the case $\cP=(\Z,\Z,\Z_{(p)})$ applied to
$\widehat\Gamma$.  (In this case we define and use
$\cP^n \widehat\Gamma$ only for $n\le 3$.)

\begin{lemma}
  \label{lemma:computation-mixed-coeffcient-series}
  Suppose $p$ is a prime.  For $\cP=(\Z,\Z,\Z_{(p)})$, the following
  hold:
  \begin{enumerate}
  \item $\cP^1 \widehat\Gamma$ is the subgroup of $\widehat\Gamma$
    generated by $\Z_{(2)}/\Z$ and $(2\Z_{(2)}^t)^2$, and
    $\widehat\Gamma/\cP^1\widehat\Gamma \cong (\Z/2)^2\times \Z$.
  \item $\cP^2 \widehat\Gamma$ is the subgroup $\Z_{(2)}/\Z$, and
    $\widehat\Gamma/\cP^2\widehat\Gamma \cong (\Z_{(2)}^t)^2\rtimes
    \Z$.
  \item $\cP^3 \widehat\Gamma$ is the subgroup $(\Z_{(2)} \cap
    \Z_{(p)})/\Z$, of $\cP^2\widehat\Gamma= \Z_{(2)}/\Z$.  Also,
    \[
    \cP^2\widehat\Gamma/\cP^3\widehat\Gamma \cong 
    \begin{cases}
      0 & \text{if }p=2,\\
      \Z[\frac{1}{p}]/\Z &\text{otherwise}.
    \end{cases}
    \]
    Consequently, if $p=2$, then $\widehat\Gamma/\cP^3\widehat\Gamma
    \cong (\Z_{(2)}^t)^2\rtimes \Z$.  If $p$ is odd, then we have the following central extension:
    \[
    1\to \Z[\tfrac{1}{p}]/\Z \to
    \widehat\Gamma/\cP^3\widehat\Gamma \to (\Z_{(2)}^t)^2\rtimes \Z
    \to 1
    \]
  \end{enumerate}
\end{lemma}

\begin{proof}
  From our choice of $\cP$, it follows that $\cP^k(-)$ is the ordinary
  derived series for $k\le 2$.

  In the proof of Theorem~\ref{theorem:computation-of-localization},
  we have checked that $\cP^1 \Gamma_{2k-1} =
  [\Gamma_{2k-1},\Gamma_{2k-1}]$ is the subgroup generated by
  $x^2$,~$y^2$.  Therefore, by the direct limit construction of
  $\widehat\Gamma$, it follows that $\cP^1 \widehat\Gamma =
  [\widehat\Gamma,\widehat\Gamma]$ is the subgroup generated by
  $\Z_{(2)}/\Z$ and $(2\Z_{(2)}^t)^2$, and $\widehat\Gamma/\cP^1
  \widehat\Gamma \cong (\Z_{(2)}^t/2\Z_{(2)}^t)^2\rtimes \Z \cong
  (\Z/2)^2\times \Z$.  This proves~(1).

  Now, again from the presentation of $\Gamma_{2k-1}$,
  $\cP^2\Gamma_{2k-1}=[\cP^1\Gamma_{2k-1},\cP^1\Gamma_{2k-1}]$ is
  generated by $[x^2,y^2]=[x,y]^4$, which is equal to the subgroup
  generated by $[x,y]$ since $[x,y]$ has odd order in~$\Gamma_{2k-1}$.
  Therefore
  $\cP^2\widehat\Gamma=[\cP^1\widehat\Gamma,\cP^1\widehat\Gamma]$ is
  seen to be the subgroup $\Z_{(2)}/\Z$, and
  $\widehat\Gamma/\cP^2\widehat\Gamma \cong (\Z_{(2)}^t)^2\rtimes \Z$.
  This proves~(2).

  For (3), recall that $\cP^3\widehat\Gamma$ is the kernel of the map
  $\Z_{(2)}/\Z \to (\Z_{(2)}/\Z) \otimes \Z_{(p)}$.  It follows
  immediately that
  \begin{align*}
    \cP^3\widehat\Gamma &= \{\text{elements in $\Z_{(2)}/\Z$ with
      order coprime to $p$}\}
    \\
    &= \{ b/a+\Z \in \Z_{(2)}/\Z \mid (a,b)=1=(a,p)\} = ( \Z_{(2)} \cap
    \Z_{(p)} ) / \Z.
  \end{align*}
  Also $(\Z_{(2)}/\Z)/\cP^3\widehat\Gamma = \Z_{(2)}/
  (\Z_{(2)}\cap\Z_{(p)})$.  One easily verifies that for any odd
  prime $p$ the inclusion of rings $\Z[\frac{1}{p}] \to
  \Z_{(2)}$ gives rise to an isomorphism
  \[
  \Z[\tfrac{1}{p}]/\Z \cong
  \Z_{(2)}/(\Z_{(2)}\cap \Z_{(p)}).
  \]
  Since $\widehat\Gamma/\cP^3\widehat\Gamma$ is a central extension of
  $\cP^2\widehat\Gamma/\cP^3\widehat\Gamma$ by
  $\widehat\Gamma/\cP^2\widehat\Gamma$, we obtain the exact sequence
  stated in~(3).
\end{proof}

\begin{remark}
  \label{remark:order-in-mixed-coefficient-quotients}
  From the above computation, it follows that an element $g\in
  \widehat\Gamma$ has finite order if and only if $g\in \cP^2
  \widehat\Gamma = \Z_{(2)}/\Z$.  If $p$ is odd, the quotient map
  $\cP^2\widehat\Gamma \to \cP^2\widehat\Gamma/\cP^3\widehat\Gamma$ is
  exactly the map eliminating all the $q$-primary factors of
  $\Z_{(2)}/\Z$ for which $(p,q)=1$.  In particular, if $g\in
  \widehat\Gamma$ has order $p^kr$ with $(p,r)=1$, then its image in
  $\widehat\Gamma/\cP^3\widehat\Gamma$ has order~$p^k$.
\end{remark}

\begin{lemma}
  \label{lemma:mixed-coefficient-amenable-in-D(Z/p)}
  The group $\widehat\Gamma/\cP^3\widehat\Gamma$ is amenable and
  in~$D(\Z/p)$.
\end{lemma}

\begin{proof}
  By the exact sequence in
  Lemma~\ref{lemma:computation-mixed-coeffcient-series}~(3), it
  follows that $\widehat\Gamma/\cP^3\widehat\Gamma$ admits a subnormal
  series with successive quotients either torsion-free abelian or
  $p$-torsion abelian, namely, $\Z$, $(\Z_{(2)}^t)^2$, and
  $\Z[\frac{1}{p}]/\Z$.  By applying Strebel's work [Str74] and known
  properties of amenable groups (see
  \cite[Lemma~6.8]{Cha-Orr:2009-01}), we obtain the desired
  conclusion.
\end{proof}

In this paper we mainly use the case of odd prime~$p$.  We remark that
if $p=2$, then $\widehat\Gamma/\cP^3\widehat\Gamma \cong
\widehat\Gamma/\cP^2\widehat\Gamma$ is PTFA.

\subsection{Homology cobordism invariants from mixed-coefficient
  commutators}
\label{subsection:homology-cobordism-invariant}

Recall that two closed 3-manifolds $M$ and $N$ are (topologically)
\emph{$R$-homology cobordant} by a topological 4-manifold \emph{$R$-homology cobordism} $W$
if $\partial W=M\cup -N$ and 
$$
H_*(W,M;R)=0=H_*(W,N;R).
$$  
We give a special case of a key result in~\cite{Cha-Orr:2009-01}:

\begin{theorem}
[Cha-Orr~\protect{\cite[Special case of Theorem~1.1]{Cha-Orr:2009-01}}]
  \label{theorem:cha-orr-L2-signature}
  Suppose $R$ is a subring of $\Q$ or~$\Z/p$.  Suppose $W$ is an
  $R$-homology cobordism between closed 3-manifolds $M$ and~$N$.
  Suppose $G$ is an amenable group lying in $D(R)$ and $\phi\colon
  \pi_1(M) \to G$ and $\psi\colon \pi_1(N)\to G$ are homomorphisms
  that extend to a common homomorphism $\pi_1(W)\to G$.  Then, the
  $L^2$-Betti numbers $\bt_i(M,\phi)$ and $\bt_i(N,\psi)$ are equal
  for any~$i$, and $L^2$-signature defects $\rhot(M,\phi)$ and
  $\rhot(N,\phi)$ are equal.
\end{theorem}

For the definitions of $\bt_i(M,\phi)$ and $\rhot(M,\phi)$, see, e.g.,
\cite[Section~7]{Cha-Orr:2009-01}.

Combining the $\Z/p$ version of Theorem~\ref{theorem:cha-orr-L2-signature} with our computation, we obtain the following result:

\begin{theorem}
  \label{theorem:homology-cobordism-invariant-over-gamma-hat}
  For a closed 3-manifold $M$ satisfying $\widehat{\pi_1(M)}\cong
  \widehat\Gamma$, let $\phi_M$ be the composition
  \[
  \phi_M\colon \pi_1(M)\to \widehat{\pi_1(M)} \cong \widehat\Gamma \to
  \widehat\Gamma / \cP^3 \widehat\Gamma
  \]
  where $\cP=(\Z,\Z,\Z_{(p)})$.  Then the $i$-th Betti number
  $\bt_i(M,\phi_M)$ and $L^2$-signature defect $\rhot(M,\phi_M)$ are
  homology cobordism invariants.  That is, if another 3-manifold $N$
  is homology cobordant to $M$, then $\widehat{\pi_1(N)}\cong
  \widehat\Gamma$, and we have $\bt_i(M,\phi_M) = \bt_i(N,\phi_N)$,
  and $\rhot(M,\phi_M) = \rhot(N,\phi_N)$.
\end{theorem}

\begin{proof}
  Suppose $M$ and $N$ are as in the statement of the theorem, and $W$
  is a homology cobordism between $M$ and~$N$.  Then, since
  $\pi_1(M)\to \pi_1(W)$ and $\pi_1(N) \to \pi_1(W)$ are integral
  homology 2-connected, we have $\widehat\Gamma \cong
  \widehat{\pi_1(M)}\cong \widehat{\pi_1(W)} \cong
  \widehat{\pi_1(N)}$.  Also it follows that $\phi_M$ and $\phi_N$ are
  restrictions of the composition
  \[
  \pi_1(W)\to \widehat{\pi_1(W)} \cong \widehat\Gamma \to
  \widehat\Gamma / \cP^3 \widehat\Gamma.
  \]
  Since $W$ is also a $(\Z/p)$-homology cobordism and $\widehat\Gamma
  / \cP^3 \widehat\Gamma$ is in $D(\Z/p)$ by
  Lemma~\ref{lemma:mixed-coefficient-amenable-in-D(Z/p)}, the
  $(\Z/p)$-version of Theorem~\ref{theorem:cha-orr-L2-signature}
  implies
  \begin{equation*}
    \bt_i(M,\phi_M) = \bt_i(N,\phi_N) \quad\mbox{and}\quad
    \rhot_i(M,\phi_M) = \rhot_i(N,\phi_N).
    \qedhere
  \end{equation*}

\end{proof}

\section{Group localization and homology cobordism of 3-manifolds}
\label{section:realization-elts-in-3mfd-groups}

In this section, $R$ is always a subring of~$\Q$.

\begin{theorem}
  \label{theorem:construction-of-3-manifold}
  Suppose $M_0$ is a closed 3-manifold.  If $\alpha\colon\pi_1(M_0) \to
  G$ is a homomorphism into a finitely presented group $G$ which is
  $R$-homology 2-connected, then there is an $R$-homology cobordism
  $W$ between $M_0$ and a closed hyperbolic 3-manifold $M$ satisfying
  the following:
  \begin{enumerate}
  \item $\alpha$ extends to a surjection $\pi_1(W)\to G$.
  \item The inclusion induces a surjection $\pi_1(M) \to \pi_1(W)$.
  \end{enumerate}
  Consequently, the induced map $\pi_1(M) \to G$ is a surjection which
  is $R$-homology 2-connected.
\end{theorem}

It is known that for any finitely presented group $G$, the
$R$-homology localization $\widehat G$ is the direct limit of a
sequence of homomorphisms on finitely presented groups
\[
G=G_0 \to G_1 \to \cdots \to G_n \to \cdots
\]
which are $R$-homology 2-connected \cite[Theorem~2.6]{Cha:2004-1}.
From this we obtain the following consequence of
Theorem~\ref{theorem:construction-of-3-manifold}:

\begin{corollary}
  \label{corollary:construction-of-3-manifold}
  Suppose $M_0$ is a closed 3-manifold with $G=\pi_1(M_0)$.  For any
  finitely generated subgroup $H$ in the $R$-homology localization
  $\widehat G$, there is a closed hyperbolic 3-manifold $M$ satisfying
  the following:
  \begin{enumerate}
  \item $M$ is $R$-homology cobordant to~$M_0$.
  \item There is a 2-connected homomorphism $f\colon \pi_1(M) \to G$
    such that
    \[
    H \subset \Im \{\pi_1(M) \to \widehat{\pi_1(M)}
    \xrightarrow[\cong]{f} \widehat G\}.
    \]
 \end{enumerate}
\end{corollary}

\subsection{Algebraic equations over groups}

To prove Theorem~\ref{theorem:construction-of-3-manifold}, we will
construct a homology cobordism of $M_0$ by attaching 1- and 2-handles
according to a homological description encoded through certain equations over the group $G$.
For this purpose, we first recall a definition from
\cite[Definition~4.1]{Cha:2004-1}.  We denote by $d(R)$ the set of
denominators of reduced fractional expressions of elements in~$R$.

\begin{definition}
  A \emph{system of $R$-nullhomologous equations} over a group $\pi$ in
  $n$ variables $x_1,\ldots,x_n$ is a collection $S$ of $n$
  expressions $x_i^e=w_i(x_1,\ldots,x_n)$, $i=1,\ldots,n$, where $e\in
  d(R)$ and each $w_i=w_i(x_1,\ldots,x_n)$ is an element in the free
  product of $G$ and the free group $F$ generated by the $x_i$ which
  lies in the kernel of the projection $G*F \to F \to H_1(F)=F/[F,F]$.
\end{definition}

Levine first defined his group closure in~\cite{Levine:1989-1,Levine:1989-2}, to extend link invariants in~\cite{Orr:1989-1}.  Levine's group closure gave an alternative description of the group localization of Vogel, and was inspired, in part, from a similar notion of group closure defined by Guti\'errez in~\cite{Gutierrez:1979-1}.  Shortly thereafter, Farjoun and Shelah rediscovered an analogous construction to describe Bousfield homology localization of groups~\cite{Farjoun-Orr-Shelah:1989}.  Variations, for instance, with untwisted and twisted coefficient systems, appear in various sources, for instance,~\cite{Cha:2004-1} and~\cite{Heck:2009-1}.

For a system of nullhomologous equations
$S=\{x_i^e=w_i(x_1,\ldots,x_n)\}$ over $\pi$, we associate a group
$\pi_S$ defined as follows:
\[
\pi_S = \langle \pi, z_1,\ldots,z_n \mid z_i^e = w_i(z_1,\ldots,z_n),\
i=1,\ldots,n \rangle
\]
This group $\pi_S$ is endowed with a natural map $\pi \to \pi_S$ and
can be viewed as obtained from $\pi$ by adjoining a solution $\{z_i\}$
of the system~$S$.  We need the following fact:

\begin{lemma}
  \label{lemma:extension-by-equation}
  \leavevmode\Nopagebreak
  \begin{enumerate}
  \item For any $\pi$ and any system of $R$-nullhomologous equations
    $S$ over $\pi$, the map $\pi\to \pi_S$ is $R$-homology
    2-connected.
  \item If $\alpha\colon\pi\to G$ is a homomorphism which is
    $R$-homology 2-connected and $G$ is finitely generated, then
    there is a system of $R$-nullhomologous equations $S$ over $\pi$
   and a surjection $\alpha_S\colon \pi_S \to G$
    making the following diagram commute:
    \[
    \xymatrix{
      \pi \ar[r]\ar[d]_{\alpha}
      & \pi_S \ar[dl]^{\alpha_S}
      \\
      G
    }
    \]
  \end{enumerate}
\end{lemma}

For a proof of Lemma~\ref{lemma:extension-by-equation}, see
\cite[p.~245, proof of Theorem~5.2]{Cha:2004-1}.  (See also \cite{Bousfield:1977}, Corollary~2.17, for a related result to part (2) above.)

\subsection{Construction of an homology cobordism using equations}

We construct a homology cobordism
from $R$-null\-homologous equations.

\begin{proof}[Proof of Theorem~\ref{theorem:construction-of-3-manifold}]
  Suppose $M_0$ is a closed 3-manifold with $\pi=\pi_1(M_0)$, and
  $\alpha\colon \pi \to G$ is $R$-homology 2-connected.
  By
  Lemma~\ref{lemma:extension-by-equation}, we obtain a system of
  $R$-nullhomologous equations $S=\{x_i^e = w_i\}_{i=1,\ldots,n}$ over
  $\pi$ such that the given $\alpha$ extends to a surjection $\pi_S\to
  G$.  We start with $M_0\times[0,1]$.  Let $V_1$ be the cobordism
  from $M_0$ to another 3-manifold, say $M_1$, which is obtained by
  attaching $n$ 1-handles to $M_0\times[0,1]$, one for each
  variable~$x_i$.  Identify $\pi_1(V_1)$ with the free product of
  $\pi$ and the free group generated by the~$x_i$.  Since $M_1$ has
  dimension 3, we can choose a collection of disjoint simple closed
  curves $\gamma_i$ on $M_1$ which represent the elements $x_i^{-e}w_i
  \in \pi_1(V_1)$.  By attaching $n$ 2-handles along the $\gamma_i$,
  we obtain a cobordism, say $V$, between $M_0$ and a new 3-manifold,
  say~$N$.

  We will verify the following properties of $(V,M_0,N)$ constructed
  above:
  \begin{enumerate}
  \item[(0)] $V$ is an $R$-homology cobordism between $M_0$ and~$N$.
  \item[(1)] The given $\alpha\colon \pi_1(M_0) \to G$ extends to a
    surjection $\pi_1(V) \to G$.
  \item[(2)] The inclusion induces a surjection $\pi_1(N) \to
    \pi_1(V)$.
  \end{enumerate}
  \[
  \xymatrix{
    \pi_1(M_0) \ar[r] \ar[dr]_{\alpha}
    & \pi_1(V) \ar@{..>>}[d]
    & \pi_1(N) \ar@{->>}[l]
    \\
    & G
  }
  \]

  First note that $H_*(V,M_0;R)$ can be computed from the handlebody
  structure of~$V$: the cellular chain complex vanishes in dimensions
  other than 1 and 2, and the image of the $i$th 2-handle under the
  boundary map $\partial_1$ is determined by the word $x_i^{-e}w_i$.
  Since the equations are nullhomologous, it follows that $\partial_1$
  is represented by an $n\times n$ diagonal matrix with all diagonals
  $-e$.  Since $e$ is a unit in $R$, it follows that $H_*(V,M_0;R)=0$.
  By duality, $H_*(V,N;R)=0$.  That is, $V$ is a homology cobordism
  between $M_0$ and~$N$.

  From the construction, $\pi_1(V)$ can be identified with $\pi_S$.
  By our choice of the equations (i.e., by
  Lemma~\ref{lemma:extension-by-equation}), $\alpha$ extends to a
  surjection $\pi_1(V)=\pi_S \to G$.  

  Reversing the handle decomposition, $V$ is obtained from $N$ by
  attaching 2- and 3-handles, and therefore $\pi_1(N)\to \pi_1(V)$ is
  surjective.  This completes the verification of the above properties
  (0), (1),~(2).

  Note that our $N$ is not necessarilly hyperbolic.  To obtain a
  hyperbolic 3-manifold satisfying the same properties, we need the
  following result of Ruberman:

  \begin{theorem}[Ruberman~\protect{\cite[Theorem~2.6]{Ruberman:1990-1}}]
    \label{theorem:H-cob-to-hyperbolic-3mfd}
    For any closed 3-manifold $N$, there is a homology cobordism $U$
    between $N$ and a closed hyperbolic 3-manifold $M$ and a
    retraction $r\colon U\to N$ such that the composition $M
    \hookrightarrow U \xrightarrow{r} N$ is a degree one map.
  \end{theorem}

  Applying Theorem~\ref{theorem:H-cob-to-hyperbolic-3mfd} to our $N$
  and by attaching the resulting homology cobordism to the above
  cobordism $V$, we obtain a new $R$-homology cobordism, say $W$,
  between $M_0$ and the resulting hyperbolic 3-manifold~$M$.  Since a
  degree one map induces a surjection on the fundamental group
  \cite[Lemma~15.12]{Hempel:1976-1}, $(W,M_0,M)$ satisfies the above
  properties (0), (1),~(2).
  
  The last sentence of
  Theorem~\ref{theorem:construction-of-3-manifold} follows from the
  above properties since the induced maps of $\pi_1(M)$, $\pi_1(M_0)$
  into $\pi_1(W)$ are all $R$-homology 2-connected, and therefore
  induce isomorphisms on the $R$-homology localization.
\end{proof}

\section{3-manifolds with local hidden torsion}
\label{section:3mfds-with-hidden-local-torsion}

Throughout this section, we fix an odd prime~$p$.  Several objects
constructed in this section depend on $p$, but for simplicity, we omit
$p$ in our notation.

\subsection{Construction of examples}
\label{subsection:construction-of-examples}

\subsubsection*{Construction of a ``seed'' 3-manifold}

Let $Y$ be the twisted torus bundle over $S^1$ with fundamental
group~$\Gamma$ defined in
Section~\ref{section:computation-for-torus-bundle-group}.  (For a
description of $Y$, see the beginning of
Section~\ref{section:computation-for-torus-bundle-group}.)  Recall
that the $\Z$-homology localization $\widehat \Gamma$ is given as the
central extension
\[
0\to \Z_{(2)}/\Z \to \widehat\Gamma \to (\Z_{(2)}^t)^2 \rtimes \Z \to 0.
\]
Let $M$ be a hyperbolic 3-manifold obtained by applying
Corollary~\ref{corollary:construction-of-3-manifold} to $Y$ and the
subgroup $H$ generated by $\frac{1}{p}+\Z \in \Z_{(2)}/\Z \subset
\widehat\Gamma$.  Then our $M$ is homology cobordant to~$Y$, and for
$\pi=\pi_1(M)$, there is a 2-connected surjection $\pi\to \Gamma$
inducing $\widehat\pi\cong \widehat\Gamma$.  Under this isomorphism we
identify $\widehat\pi$ with $\widehat\Gamma$; then by our choice of
$M$, $\frac{1}{p}+\Z \in \widehat\Gamma=\widehat\pi$ is the image of
some nontrivial element $[\alpha]\in \pi=\pi_1(M)$, where $\alpha$ is
a simple closed curve in~$M$.  Since $\pi$ is torsion-free, $[\alpha]$
has infinite order in $\pi$ but its image in
$\widehat\pi=\widehat\Gamma$ has order~$p$.  That is, $[\alpha]\in
\pi$ is a local hidden torsion.

Since $\Z_{(2)}/\Z\subset \widehat\Gamma=\widehat\pi$ lies in
$\widehat\pi_\omega$ by
Lemma~\ref{lemma:omega-lower-central-series-gamma-hat}, the pre-image
of $\Z_{(2)}/\Z$ in $\pi$ lies in $\pi_\omega$ by Stallings'
Theorem~\cite{Stallings:1965-1} $\pi/\pi_\omega \hookrightarrow
\widehat\Gamma/\widehat\Gamma_\omega$.
  Furthermore, $[\alpha]$ is not contained in $\pi_{\omega+1}$ since
  the image of $[\alpha]$ in $\widehat\Gamma$ is nontrivial and
  $\widehat\Gamma_{\omega+1}$ is trivial by
  Lemma~\ref{lemma:omega-lower-central-series-gamma-hat}.
Summarizing, we have proven the following:

\begin{lemma}
  \label{lemma:hidden-local-torsion-in-seed-mfd}
  The element $[\alpha]$ chosen above is local hidden torsion of
  $\pi_1(M)$ and lies in~$\pi_1(M)_\omega-\pi_1(M)_{\omega+1}$.
\end{lemma}

 \begin{remark}   It follows that $\pi_1(M)$ has lower central series length
  $>\omega$.  Since $M$ is homology cobordant to $Y$ and
  $\Gamma=\pi_1(Y)$ has lower central series length $\omega$, this
  proves Theorem~\ref{theorem:answer-to-cochran-freedman} stated in
  the introduction.
  \end{remark}

\subsubsection*{Twisting around local hidden torsion}

For a knot $K$ in $S^3$, we define $M(\alpha,K)$ to be the 3-manifold
obtained by removing a tubular neighborhood of $\alpha$ from $M$ and
then filling it in with the exterior of $K$ in such a way that a
meridian of $K$ is identified with a parallel copy of $\alpha$ and a
longitude of $K$ is identified with a meridian curve of the tubular
neighborhood.  It is well known that the resulting $M(\alpha,K)$ is
always homology equivalent to~$M$.  Since we have to use it again
later in this paper, we state it as a lemma.

\begin{lemma}[Well-known]
  \label{lemma:infection-homology-equivalence}
  For a knot $K$, there is a homology equivalence $h_K\colon
  M(\alpha,K)\to M$ which induces a surjection ${h_K}_*\colon
  \pi_1(M(\alpha,K))\to \pi_1(M)$.
\end{lemma}

In fact, for the exterior $E_K$ of a knot $K$, there is a homology
equivalence $(E_K,\partial E_K) \to (S^1\times D^2, S^1\times S^1)$
which induces a surjection on~$\pi_1(-)$, and our $h_K$ is obtained by
glueing this with the identity map of the exterior of $\alpha\subset
M$ (e.g., see \cite[Proof of Proposition~4.8]{Cha:2007-1}).

Now, applying Theorem~\ref{theorem:H-cob-to-hyperbolic-3mfd} to our
$M(\alpha,K)$, we obtain a hyperbolic 3-manifold $\Hyp_K$ homology
cobordant to $M(\alpha,K)$ and a homology equivalence $g_K\colon
\Hyp_K \to M(\alpha,K)$ inducing a surjection on~$\pi_1(-)$.  Note
that if $K$ is unknotted, we may assume that $\Hyp_K$ is equal to~$M$.

\begin{proposition}
  \label{proposition:common-properties-of-examples}
  For any knot $K$, the following hold:
  \begin{enumerate}
  \item There is a homology equivalence $f_K\colon \Hyp_K \to M$ which
    induces a surjection on the fundamental group.
  \item The homology localization of $\pi_1(\Hyp_K)$ is isomorphic to
    $\widehat\Gamma$.
  \item There are local hidden torsion elements of $\pi_1(\Hyp_K)$ lying
    in $\pi_1(\Hyp_K)_\omega$.
 \end{enumerate}
\end{proposition}

\begin{proof}
  (1) Composing $g_K$ with $h_K$ given in
  Lemma~\ref{lemma:infection-homology-equivalence}, we obtain a
  desired homology equivalence $f_K\colon \Hyp_K\to M$.

  (2) Since ${f_K}_*\colon \pi_1(\Hyp_K) \to \pi_1(M)$ is 2-connected,
  $\widehat{\pi_1(\Hyp_K)} \cong \widehat{\pi_1(M)} = \widehat\Gamma$.

  (3) Since ${f_K}_*$ is surjective, there is an element, say
  $[\beta]$, in $\pi_1(\Hyp_K)$ which is sent to
  $[\alpha]\in\pi_1(M)$.  Since $\pi_1(\Hyp_K)$ is torsion-free and
  $[\alpha]$ is a local hidden torsion element by
  Lemma~\ref{lemma:hidden-local-torsion-in-seed-mfd}, it follows that
  $[\beta]$ is a local hidden torsion element in~$\pi_1(\Hyp_K)$.
  Since $f_K$ induces an isomorphism
  $\pi_1(\Hyp_K)/\pi_1(\Hyp_K)_\omega \cong \pi_1(M)/\pi_1(M)_\omega$
  by Stallings' Theorem~\cite{Stallings:1965-1} and $[\alpha] $ lies
  in $\pi_1(M)_\omega$ by
  Lemma~\ref{lemma:hidden-local-torsion-in-seed-mfd}, we have
  $[\beta]\in \pi_1(\Hyp_K)_\omega$.
\end{proof}

Now, Theorem~\ref{theorem:intro-hidden-local-torsion} in the
introduction is an immediate consequence of
Proposition~\ref{proposition:common-properties-of-examples}.

\subsection{Homology cobordism classes and amenable $L^2$-signatures}
\label{subsection:examples-with-distict-homology-cobordism-classes}

We continue to use the notations of the previous subsection.  In
particular $p$ is a fixed odd prime and $M$ and $\Hyp_K$ are the
hyperbolic 3-manifolds given in
Section~\ref{subsection:construction-of-examples}.

We consider the $L^2$-signature invariants discussed in
Section~\ref{subsection:homology-cobordism-invariant}: recall that for
any 3-manifold $N$ with $\widehat{\pi_1(N)}\cong \widehat\Gamma$, we
define $\phi_N$ to be the composition
\[
\phi_N\colon \pi_1(N) \to \widehat{\pi_1(N)} \to
\widehat{\pi_1(N)}/\cP^3\widehat{\pi_1(N)} \cong
\widehat\Gamma/\cP^3 \widehat\Gamma
\]
where $\cP^n G$ denotes the mixed-coefficient commutator series
associated to $\cP=(\Z,\Z,\Z_{(p)})$.

Using the key fact that the image of $[\alpha]$ in the homology
localization has finite order $p$, we show the following formula for
the $L^2$-signature of~$\Hyp_K$.  We denote the Levine-Tristram signature
function of a knot $K$ by
\[
\sigma_K(w)=\sign (1-w)A+(1-\overline w)A^T
\]
for $w\in S^1\subset \C$, where $A$ is a Seifert matrix for~$K$.

\begin{lemma}
  \label{lemma:computation-of-L2sign}
  For our $M$ and $\Hyp_K$, the following holds:
  \[
  \rhot(\Hyp_K,\phi_{\Hyp_K}) = \rhot(M,\phi_M) + \frac1p \sum_{k=0}^{p-1}
  \sigma_K(e^{2\pi k\sqrt{-1}/p})
  \]
\end{lemma}

\begin{proof}
  By
  Theorem~\ref{theorem:homology-cobordism-invariant-over-gamma-hat},
  $\rhot(N,\phi_N)$ is a homology cobordism invariant.  Therefore, in
  our case, we may replace $(\Hyp_K,\phi_{\Hyp_K})$ by
  $(M(\alpha,K),\phi_{M(\alpha,K)})$ since $\Hyp_K$ is homology cobordant
  to~$M(\alpha,K)$.

  By \cite[Proposition~3.2]{Cochran-Orr-Teichner:2002-1} (see also
  \cite[Lemma~2.3]{Cochran-Harvey-Leidy:2009-1}), for any $\phi\colon
  \pi_1(M)\to G$, we have
  \[
  \rhot(\Hyp_K,\phi\circ {h_K}_*) = \rhot(M,\phi) + \rhot(N_K,\phi')
  \]
  where $N_K$ is the zero-surgery manifold of $K$, $\phi'\colon
  \pi_1(N_K) \to G$ is the map that factors through $H_1(N_K)=\Z$ and
  sends any meridian of $K$ to~$\phi([\alpha])$, and $h_K\colon
  M(\alpha,K) \to M$ is the homology equivalence given in
  Lemma~\ref{lemma:infection-homology-equivalence}.

  For our purpose, consider the case of $\phi=\phi_M$.  Then $\phi_M
  \circ {h_K}_*$ is equal to $\phi_{\Hyp_K}$ (up to an automorphism of
  $\widehat\Gamma/\cP^3\widehat\Gamma$) since $h_K$ gives rise to an
  isomorphism on $\widehat{\pi_1(-)}$.  Therefore
  $\rhot(\Hyp_K,\phi_M\circ {h_K}_*) = \rhot(\Hyp_K,\phi_{\Hyp_K})$.

  Also, since the image of $[\alpha] \in \pi_1(M)$ in
  $\widehat{\pi_1(M)} = \widehat\Gamma$ has order $p$,
  $\phi_M([\alpha])$ has order $p$ in
  $\widehat\Gamma/\cP^3\widehat\Gamma$ by
  Lemma~\ref{lemma:computation-mixed-coeffcient-series}.  (See also
  Remark~\ref{remark:order-in-mixed-coefficient-quotients}.)
  Therefore, by the $L^2$-induction property and
  \cite[Lemma~8.7]{Cha-Orr:2009-01}, we obtain
  \[
  \rhot(N_K,\phi') = \frac{1}{p}\sum_{k=0}^{p-1} \sigma_K(e^{2\pi
    k\sqrt{-1}/p}).
  \]
  From these the conclusion follows.
\end{proof}

\begin{theorem}
  \label{theorem:non-homology-cobordant-examples}
  There is an infinite sequence of knots $K_0 =$ unknot, $K_1$,
  $K_2,\ldots$ in $S^3$, for which
  \begin{enumerate}
  \item $\int_{S_1} \sigma_{K_i}(w)\,dw=0$ for any $i$, and
  \item $\Hyp_{K_i}$ and $\Hyp_{K_j}$ are not homology cobordant whenever
    $i\ne j$.
  \end{enumerate}
\end{theorem}

We remark that the conclusion on the signature function integral will
be used in the next subsection.

\begin{proof}
  By \cite[Proposition 4.12]{Cha:2010-01}, there exists a knot $K$
  satisfying
  \[
  \sum_{k=0}^{p-1} \sigma_K(e^{2k\pi\sqrt{-1}/p}) \ne 0
  \quad\text{and}\quad \int_{S_1} \sigma_K(w)\,dw=0. \tag{$*$}
  \]
  Furthermore, we can choose an infinite sequence of knots $K_1, K_2,
  \ldots$ for which $(*)$ holds and the values of the sum
  $\sum_{k=0}^{p-1} \sigma_{K_i}(e^{2k\pi\sqrt{-1}/p})$ are all
  nonzero and mutually distinct.  For example, we may choose as $K_n$
  the connected sum of $n$ copies of a fixed knot $K$
  satisfying~$(*)$.

  By Lemma~\ref{lemma:computation-of-L2sign}, we have
  \[
  \rhot(\Hyp_{K_i},\phi_{\Hyp_{K_i}}) = 
  \begin{cases}
    \rhot(M,\phi_M) &\text{if }i=0, \\
    \rhot(M,\phi_M) + \frac1p \sum_{k=0}^{p-1}
    \sigma_K(e^{2k\pi\sqrt{-1}/p}) & \text{otherwise}.
  \end{cases}
  \]
  By our choice of the $K_i$, $\rhot(\Hyp_{K_i},\phi_{\Hyp_{K_i}}) \ne
  \rhot(\Hyp_{K_j},\phi_{\Hyp_{K_j}})$ whenever $i\ne j$.  By
  Theorem~\ref{theorem:homology-cobordism-invariant-over-gamma-hat},
  it follows that $\Hyp_{K_i}$ and $\Hyp_{K_j}$ are not homology cobordant
  whenever $i\ne j$.
\end{proof}  

From Proposition~\ref{proposition:common-properties-of-examples} and
Theorem~\ref{theorem:non-homology-cobordant-examples}, it follows that
the hyperbolic 3-manifolds $\Hyp_{K_i}$ in
Theorem~\ref{theorem:non-homology-cobordant-examples} satisfy the
conclusions (1) and (2) of Theorem~\ref{theorem:intro-main} in the
introduction.  The remaining conclusions of
Theorem~\ref{theorem:intro-main}, namely the fact that all the
$\Hyp_{K_i}$ look the same to the eyes of previously known homology
cobordism invariants, will be shown in the next subsection.

\subsection{Invariants from $p$-groups, iterated $p$-covers, and PTFA
  groups}
\label{subsection:p-and-PTFA-invariants}

In this subsection we consider known homology cobordism invariants
obtained from $p$-groups, iterated $p$-covers, and PTFA groups.  We
start by recalling known invariants associated to characters and
representations.

\begin{itemize}
\item For a character $\phi\colon \pi_1(M)\to \Z_d$ of a closed
  3-manifold $M$, the multisignature (= Casson-Gordon invariant =
  Atiyah-Singer $\alpha$-invariant) is defined in
  \cite{Wall:1970-1,Casson-Gordon:1978-1,Atiyah-Singer:1968-3}.  We
  denote it by $\sigma(M,\phi)$.
\item More generally, for a finite dimensional unitary representation
  $\theta\colon\pi_1(M)\to U(k)$, Atiyah-Patodi-Singer's (reduced)
  $\eta$-invariant $\tilde\eta(M,\theta)$ is defined
  in~\cite{Atiyah-Patodi-Singer:1975-2}.
\item For a representation $\theta\colon \pi_1(M)\to
  \operatorname{GL}(\C,k)$ and a homomorphism $\psi\colon \pi_1(M) \to
  H$ into a free abelian group $H$, a twisted torsion invariant
  $\tau^{\theta \otimes \psi}(M)$ is defined
  in~\cite{Cha-Friedl:2010-01}.
\end{itemize}

These are known to give homology cobordism invariants when the
representations factors through $p$-groups and characters are of prime
power order.  (See \cite{Gilmer-Livingston:1983-1, Ruberman:1984-1,
  Levine:1994-1, Cha-Friedl:2010-01}.)  The following homology
cobordism invariant obtained from iterated covers is also closely
related to $p$-groups:

\begin{itemize}
\item For a tower of iterated abelian $p$-covers
  \[
  M=M_0 \leftarrow M_1 \leftarrow \cdots \leftarrow M_n
  \]
  (here $M_n\to M$ is not necessarily abelian nor regular) and a
  character $\phi\colon \pi_1(M_n)\to \Z_d$ of prime power order, a
  Hirzebruch-type Witt-class-valued invariant $\lambda(M_n,\phi)\in
  L_0(\Q(\zeta_d))$ is defined in~\cite{Cha:2007-1}.
\end{itemize}

Following~\cite{Cha:2007-2}, we call a pair $(\{M_i\},\phi)$ as above
a \emph{$p$-structure} of height $n$ for~$M$.

We continue to use the notations of
Section~\ref{subsection:construction-of-examples}.  The following
theorem says that these invariants related to $p$-groups do not
distinguish the homology cobordism classes of our 3-manifolds~$\Hyp_K$.
The key fact used in the proof is that our hidden torsion element
$\alpha$ is invisible in the nilpotent completion.

\begin{proposition}
  \label{proposition:p-group-invariants-do-not-detect}
  For any knot $K$, the following hold:
  \begin{enumerate}
  \item For any prime power order character $\phi\colon \pi_1(M)\to
    \Z_{d}$, $\sigma(M,\phi)=\sigma(\Hyp_K,\phi\circ {f_K}_*)$.
  \item For any unitary representation $\theta$ of $\pi_1(M)$
    factoring through a $p$-group,
    $\tilde\eta(M,\theta)=\tilde\eta(\Hyp_K,\theta\circ {f_K}_*)$.
  \item For any representation $\theta\colon \pi_1(M)\to
    \operatorname{GL}(\C,k)$ factoring through a $p$-group and for any
    homomorphism $\psi\colon \pi_1(M) \to H$ with $H$ free abelian, we
    have $\tau^{\theta\otimes\psi}(M) = \tau^{(\theta\circ{f_K}_*)
      \otimes (\psi\circ{f_K}_*)}(\Hyp_K)$.
  \item The map $f_K$ gives rise to a 1-1 correspondences between
    $p$-structures of $M$ and $\Hyp_K$ via pullback (i.e., $f_K$ is a
    $p$-tower map in the sense of \cite{Cha:2007-1}).  If
    $(\{M_i\},\phi)$ and $(\{M_i'\},\phi')$ are correponding
    $p$-structures of height $n$ for $M$ and $\Hyp_K$, respectively, then
    $\lambda(M_n,\phi)=\lambda(M_n',\phi')$.
  \end{enumerate}
\end{proposition}

\begin{proof}
  Recall that $\Hyp_K$ is homology cobordant to $M(\alpha,K)$ via a
  homology cobordism $U$ which admits a retraction $r\colon U\to
  M(\alpha,K)$, and $f_K$ is the composition $\Hyp_K\hookrightarrow W
  \xrightarrow{r} M(\alpha,K) \xrightarrow{h_K} M$.  Since the
  invariants considered in
  Proposition~\ref{proposition:p-group-invariants-do-not-detect} are
  all homology cobordism invariants as shown in
  \cite{Gilmer-Livingston:1983-1, Ruberman:1984-1, Levine:1994-1,
    Cha:2007-1, Cha-Friedl:2010-01}, we may replace $(\Hyp_K,f_K)$ by
  $(M(\alpha,K),h_K)$.  Now, since our $[\alpha]$ lies in
  $\pi_1(M)_\omega$, the image of $[\alpha]$ under any map of
  $\pi_1(M)$ into a nilpotent group is trivial.  Therefore $[\alpha]$
  lies in the kernel of the given representation/character in (1),
  (2),~(3).  

  For the tower $\{M_i\}$ of iterated covers in (4), first observe that
  if we view $\pi_1(M_{i+1})$ as a subgroup of $\pi_1(M_i)$, then
  $\pi_1(M_i)_\omega$ is contained in $\pi_1(M_{i+1})_\omega$ since
  any $p$-group~$G$ is nilpotent.  It follows that $\alpha$ always
  lifts to a loop in $M_n$ for any choice of a basepoint in $M_n$, and
  any lift of $\alpha$ lies in the kernel of the given character
  $\phi$ of $\pi_1(M_n)$.
  
  Now, Proposition~\ref{proposition:p-group-invariants-do-not-detect}
  is an immediate consequence of the following known fact, which
  essentially says that tying a knot along a curve in the kernel of
  the given representation/character does not change the concerned
  invariants.
\end{proof}

\begin{lemma}
  \label{lemma:invariance-under-infection}
  Suppose $M$ is a closed 3-manifold, $\alpha$ is a simple closed
  curve in $M$, and $M(\alpha,K)$ is the 3-manifold obtained by tying
  a knot $K$ along~$\alpha$.  Let $h_K\colon M(\alpha,K)\to M$ be the
  homology equivalence described in
  Lemma~\ref{lemma:infection-homology-equivalence}.
  \begin{enumerate}
  \item If $\phi\colon \pi_1(M)\to \Z_d$ satisfies $[\alpha]\in \Ker
    \phi$, then $\sigma(M,\phi)=\sigma(M(\alpha,K),\phi\circ
    {h_K}_*)$.
  \item If $\theta\colon\pi_1(M)\to U(k)$ satisfies $[\alpha]\in \Ker
    \theta$, then $\tilde\eta(M,\theta) =
    \tilde\eta(M(\alpha,K),\theta\circ {h_K}_*)$.
  \item If $\theta\colon \pi_1(M)\to \operatorname{GL}(\C,k)$ and
    $\psi\colon \pi_1(M) \to H$ with $H$ free abelian satisfy
    $[\alpha] \in \Ker \theta \cap \Ker\psi$, then
    $\tau^{\theta\otimes\psi}(M) = \tau^{(\theta\circ{h_K}_*) \otimes
      (\psi\circ{h_K}_*)}(M(\alpha,K))$.
  \item The map $h_K$ gives rise to a 1-1 correspondences between
    $p$-structures of $M$ and $M(\alpha,K)$ via pullback.  If
    $(\{M_i\},\phi)$ and $(\{M_i'\},\phi')$ are correponding
    $p$-structures of height $n$ for $M$ and $M(\alpha,K)$,
    respectively, and any component of the pre-image of $\alpha$ under
    $M_n \to M$ is in $\Ker\phi$, then
    $\lambda(M_n,\phi)=\lambda(M_n',\phi')$.
  \end{enumerate}
\end{lemma}

\begin{proof}
  It is known that the multisignature obtained from $\Z_d$-valued
  characters are equivalent to $\eta$-invariants associated to
  representations factoring through $\Z_d$.  From this it follows that
  (1) is a consequence of~(2).

  The proof for (2) is similar to the arguments in
  \cite[Section~5.4]{Cha:2007-1}: for the knot tying operation, it is
  known (e.g., see \cite[Lemma~5.8]{Cha:2007-1}) that
  $\tilde\eta(\Hyp_K,\theta\circ{f_K}_*) =
  \tilde\eta(M,\theta)+\tilde\eta(N_K,\theta')$ where $N_K$ is the
  zero-surgery manifold of the knot $K$, and $\theta'\colon
  \pi_1(N_K)\to U(k)$ is the representation that factors through
  $H_1(N_K)=\Z$ and sends a meridian of $K$ to $\theta(\alpha)$.  In
  our case, $\theta'$ is the trivial representation, since
  $[\alpha]\in \pi_1(M)_\omega$ and $\pi_1(M)$ factors through a
  nilpotent group.  Therefore $\tilde\eta(N_K,\theta')=0$.

  (3) is shown by an argument similar to the above proof for (2),
  using the knot tying formula given in
  \cite[Lemma~7.1]{Cha-Friedl:2010-01} in place
  of~\cite[Lemma~5.8]{Cha:2007-1}.  For (4), by
  \cite[Lemma~3.7]{Cha:2007-1}, $h_K$ induces a 1-1 correspondence
  between $p$-structures.  Computing $\lambda(M_n',\phi')$ using the
  knot tying formula given in \cite[Lemma~4.6]{Cha:2007-1}, we obtain
  the desired conclusion.
\end{proof}

Now we show that $L^2$-signature invariants associated to torsion-free
coefficient systems do not distinguish the homology cobordism class of
$M$ from that of $\Hyp_K$ (which is equal to that of $M(\alpha,K)$).
First we consider the special case of coefficient systems factoring
through the homology localization.

\begin{proposition}
  \label{proposition:local-torsion-free-L2-sign-does-not-detect}
  \begin{enumerate}
  \item If $\phi\colon \pi_1(M) \to G$ is a homomorphism factoring
    through $\widehat{\pi_1(M)}$ and $G$ is torsion-free, then
    $\rhot(M,\phi)=\rhot(M(\alpha,K),\phi\circ {h_K}_*)$.
  \item For Harvey's homology cobordism invariant $\rho_n(-)$ defined
    in \cite{Harvey:2006-1}, we have $\rho_n(M) = \rho_n(M(\alpha,K))
    = \rho_n(\Hyp_K)$ for any $n$ and~$K$.
  \end{enumerate}
\end{proposition}

\begin{proof}
  As in the proof of Lemma~\ref{lemma:computation-of-L2sign}, we have
  \[
  \rhot(M(\alpha,K),\phi\circ {h_K}_*) = \rhot(M,\phi)+
  \rhot(N_K,\phi')
  \]
  where $N_K$ is the zero-surgery manifold of $K$ and $\phi'\colon
  \pi_1(N_K) \to G$ is the map that factors through $H_1(N_K)=\Z$ and
  sends any meridian of $K$ to~$\phi([\alpha])$.  In our case,
  $\phi([\alpha])$ is trivial because of the following facts: the
  image of $[\alpha]$ in $\widehat{\pi_1(M)}$ is torsion, $\phi$
  factors through $\widehat{\pi_1(M)}$, and $G$ is torsion-free.  That
  is, $\phi'$ is a trivial map.  It follows that $\rhot(N_K,\phi')=0$.
  This shows~(1).

  In \cite{Harvey:2006-1}, $\rho_n(M)$ is defined to be
  $\rhot(M,\phi_n)$ where $\phi_n\colon \pi_1(M)\to
  \pi_1(M)/\pi_1(M)_H^{(n)}$ is the quotient map.  Here
  $\pi_1(M)_H^{(n)}$ denotes the torsion-free derived series defined
  in~\cite{Cochran-Harvey:2004-1}.  Due to \cite{Cha:2004-1}, there is
  a canonical injection
  \[
  j\colon \pi_1(M)/\pi_1(M)_H^{(n)} \to
  \widehat{\pi_1(M)}/\widehat{\pi_1(M)}_H^{(n)}.
  \]
  By the
  $L^2$-induction property for $\rhot$, $\rho_n(M)$ is equal to
  $\rhot(M,j\circ\phi_n)$.  Since $j\circ\phi_n$ factors through
  $\widehat{\pi_1(M)}$ and $h_K$ induces an isomorphism on
  $\widehat{\pi_1(-)}$, (2) follows from~(1).
\end{proof}

\begin{proposition}
  \label{proposition:torsion-free-L2-sign-does-not-detect}
  Suppose $K$ is a knot satisfying $\int_{S^1} \sigma_K(w) \, dw = 0$.
  If $G$ is a torsion-free group, then for any $\phi\colon \pi_1(M)\to
  G$, we have $\rhot(M,\phi)=\rhot(M(\alpha,K),\phi\circ {h_K}_*)$.
\end{proposition}

\begin{proof}
  As in the proof of
  Proposition~\ref{proposition:local-torsion-free-L2-sign-does-not-detect},
  we have
  \[
  \rhot(M(\alpha,K),\phi\circ{h_K}_*) = \rhot(M,\phi)+ \rhot(N_K,\phi')
  \]
  Since $G$ is torsion-free, $\phi'$ is either trivial or onto the
  infinite cyclic subgroup generated by~$\phi([\alpha])$.  In the
  former case, $\rhot(N_K,\phi')=0$.  In the latter case,
  $\rhot(N_K,\phi')=\int_{S_1} \sigma_K(w)\,dw$ by
  \cite[Proposition~5.1]{Cochran-Orr-Teichner:2002-1}.  From this the
  conclusion follows.
\end{proof}

Now, Theorem~\ref{theorem:intro-main} in the introduction follows from
our results in this section: for the hyperbolic 3-manifolds $\Hyp_{K_i}$
given in Theorem~\ref{theorem:non-homology-cobordant-examples}, it
follows that Theorem~\ref{theorem:intro-main} (1), (2), (3), and (4)
hold, respectively, from
Proposition~\ref{proposition:common-properties-of-examples},
Proposition~\ref{proposition:torsion-free-L2-sign-does-not-detect},
Proposition~\ref{proposition:p-group-invariants-do-not-detect}, and
Theorem~\ref{theorem:non-homology-cobordant-examples}.


\appendix

\section{Local groups and local modules}
\label{section:local-groups-and-local-modules}

For completeness, we prove a result about local groups and modules used in Section~\ref{section:computation-for-torus-bundle-group} as Proposition~\ref{proposition:hat-Gamma-is-local}.

We excerpt contents in this appendix from a manuscript in preparation
by the authors~\cite{Cha-Orr:2003-01}, which gives a much more
thorough treatment of the theory of localization of spaces, groups and
modules and its role in the study of manifolds and knots.

Throughout, $R$ denotes a commutative ring with unity.

\begin{definition}
  \label{definition:local-module}
  A module $A$ over the group ring $RG$ is called a \emph{Cohn local
    module} if the following holds: for any diagram
  \[
  \begin{diagram}
    \node{F_0} \arrow{e,t}{\alpha}\arrow{s,l}{f} \node{F_1}
    \\
    \node{A}
  \end{diagram}
  \]
  with $F_0$, $F_1$ finitely generated free $RG$-modules of the same
  rank and $\alpha$ a homomorphism inducing an isomorphism $1_R\otimes
  \alpha\colon R\otimes_{RG}F \to R\otimes_{RG}F$, there is a unique
  homomorphism $g\colon F_1 \to A$ making the diagram commute, i.e.,
  $f=g\alpha$.
\end{definition}

For a discussion related to the above definition, see \cite[Section~2
and Appendix~A]{Cha-Orr:2009-01}.

\begin{theorem}
  \label{theorem:extension-of-local-group-by-local-module}
  Suppose $\Gamma$ is an $R$-local group and $A$ is a Cohn local module
  over~$R\Gamma$.  Given any extension as below,
  \[
  0\to A\to \tilde\Gamma \to \Gamma \to 1
  \]
  then $\tilde\Gamma$ is an $R$-local group.
\end{theorem}

One proves the following result, used in the proof of
Theorem~\ref{definition:local-module}, by a standard partial
chain contraction argument originally due to Vogel.  (See also~\cite{Smith:1978-1}, where a special case is proven using a different, but related, method.)

\begin{lemma}[Vogel~\cite{Vogel:1982-1}]
  \label{lemma:chain-contraction-argument}
  Suppose $A$ is a Cohn local $RG$-module and $C_*$ is a chain complex
  over $RG$ with $C_i$ finitely generated free for $i\le n$.  If
  $H_i(R\otimes_{RG} C_*)=0$ for $i\le n$, then $H_i(A\otimes_{RG}
  C_*)=0=H^i(\Hom(C_*,A))$ for $i\le n$.
\end{lemma}

Since the argument has been presented and used in several papers (e.g,
see \cite[Section~9.2]{Vogel:1982-1}, \cite[Proof of Propositon~3.2 on
p.~95]{Levine:1994-1}, \cite[Proof of
Proposition~2.10]{Cochran-Orr-Teichner:1999-1}), we omit the
proof of Lemma~\ref{lemma:chain-contraction-argument}.

\begin{proof}[Proof of
  Theorem~\protect{\ref{theorem:extension-of-local-group-by-local-module}}]

  Suppose
  \[
  0\to A \to \tilde\Gamma \xrightarrow{p} \Gamma \to 1
  \]
  is exact, where $\Gamma$ is an $R$-local group and $A$ is a Cohn
  local module over~$R\Gamma$.  Suppose $\alpha\colon \pi\to G$ is a group
  homomorphism in $\Omega^R$, that is, $\pi$ and $G$ finitely presented and $\alpha$ is $R$-homology 2-connected. Suppose $f\colon \pi
  \to \tilde\Gamma$ is given.  We will show there is a unique $g\colon
  G\to \tilde \Gamma$ satisfying $f=g\alpha$.

  Let $f'=pf$.  Since $\Gamma$ is local, there is a unique $g'\colon
  G\to \Gamma$ such that $f'=\alpha g'$.  Taking pullback of $f'$ and
  $g'$ along $\tilde\Gamma\to \Gamma$, we obtain $\tilde\pi \to
  \tilde\Gamma$ and $\tilde G\to \tilde\Gamma$ which, together with other obvious
  maps, give the following commutative diagram with exact rows:
  \[
  \begin{diagram}\dgHORIZPAD=1.2ex
    \node[2]{\hbox to0mm{\hss$0\to$\hskip.5\dgHORIZPAD}A} \arrow[2]{e}\arrow{s,=}
    \node[2]{\tilde G} \arrow[2]{e}\arrow{s,-}
    \node[2]{G\hbox to 0mm{\hskip.5\dgHORIZPAD$\to 1$\hss}}
    \arrow[2]{s,r}{g'}
    \\
    \node{\hbox to0mm{\hss$0\to$\hskip.5\dgHORIZPAD}A}
    \arrow[2]{e}\arrow{se,=}\arrow{ne,=}
    \node{}\arrow{s,=}
    \node{\tilde \pi} \arrow[2]{e}\arrow{ne}\arrow{se}
    \node{} \arrow{s}
    \node{\pi\hbox to 0mm{\hskip.5\dgHORIZPAD$\to 1$\hss}}
    \arrow{ne,t}{\alpha} \arrow{se,b}{\smash{f'}\,} \arrow{sw,b}{\smash{f}}
    \\
    \node[2]{\hbox to0mm{\hss$0\to$\hskip.5\dgHORIZPAD}A} \arrow[2]{e}
    \node[2]{\tilde \Gamma} \arrow[2]{e,b}{p}
    \node[2]{\Gamma\hbox to 0mm{\hskip.5\dgHORIZPAD$\to 1$\hss}}
  \end{diagram}
  \]

  By the universal property of pullback, $f$ gives rise to a splitting
  $s\colon \pi\to \tilde\pi$.  Similarly, homomorphisms $g\colon G\to
  \tilde\Gamma$ satisfying $f=g\alpha$ are in 1-1 correspondence with
  splittings $G\to \tilde G$ which is compatible with~$s$.
  
  Let $c \in H^2(\Gamma;A)$ be the element which correponds to the
  extension $\tilde\Gamma$ of~$\Gamma$.  Then, the above extensions
  $\tilde\pi$ and $\tilde G$ corresponds to the images of $c$ under
  the induced maps on $H^2(-;A)$:
  \[
  \begin{diagram}
    \node{H^2(\pi;A)} \node{H^2(G;A)}\arrow{w,t}{\alpha^*}
    \\
    \node{H^2(\Gamma;A)} \arrow{n,l}{f'^*} \arrow{ne,b}{g'^*}
  \end{diagram}
  \]
  Since $\tilde\pi\to \pi$ splits, $f'^*(c)=0$.  By
  Lemma~\ref{lemma:chain-contraction-argument}, $H^i(G,\pi;A)=0$ for
  $i\le 2$ since $\alpha$ is $R$-homology 2-connected.
  Therefore
  $\alpha^*$ on $H^2(-;A)$ is injective.  It follows that $g'^*(c)=0$,
  that is, $\tilde G \to G$ splits and so $\tilde G$ is a semidirect
  product of $A$ and~$G$.
 
  Now the splittings $G\to \tilde G$ are classified by $H^1(G;A)$.
  Again by Lemma~\ref{lemma:chain-contraction-argument}, $\alpha$
  gives rise to an isomorphism $H^1(G;A)\to H^1(\pi;A)$.  Therefore,
  splittings of $\tilde\pi\to \pi$ and those of $\tilde G\to G$ are in
  1-1 correspondence.  It follows that there is a unique splitting
  $G\to\tilde G$ which is compatible with $s$, and the induced
  homomorphism $g\colon G\to \tilde\Gamma$ is exactly the desired one.
  This shows that $\tilde\Gamma$ is $R$-local.
\end{proof}

For the application of
Theorem~\ref{theorem:extension-of-local-group-by-local-module} that we
needed in Section~\ref{section:computation-for-torus-bundle-group},
the following two examples of Cohn local modules are useful.

\begin{lemma}
  \label{lemma:trivial-G-module-is-RG-local}
  If $A$ is an $RG$-module with trivial $G$-action, then $A$ is a Cohn
  local module.
\end{lemma}

\begin{proof}
  From the assumption it follows that $A\cong R\otimes_{RG} A$ as
  $RG$-modules.  If $f$ and $\alpha$ are as in
  Definition~\ref{definition:local-module}, then the composition
  \[
  g\colon F_1\to R\otimes_{RG} F_1 \xrightarrow{(1_R\otimes \alpha)^{-1}}
  R\otimes_{RG}F_0 \xrightarrow{1_R\otimes f} R\otimes_{RG}A \cong
  A
  \]
  is the required unique homomorphism.
\end{proof}

Recall that $\Z^t_{(2)}$ is the abelian group $\Z_{(2)}$ with the
$\Z$-action given by negation.

\begin{lemma}
  \label{lemma:Z^t_(2)-is-local}
  Suppose $R$ is a subring of~$\Z_{(2)}$.  Then $\Z^t_{(2)}$ is a
  local $R[\Z]$-module.
\end{lemma}
\begin{proof}
  Suppose $f$ and $\alpha$ are as in 
  Definition~\ref{definition:local-module}.  Let $A(t)$ be a square
  matrix over $R[t^{\pm1}]$ representing $\alpha$, where $t$ is a
  generator of the group $\Z$.  Then $A(1)$ is invertible in $R$ since
  $\alpha$ induces an isomorphism $R\otimes_{R[\Z]}F_0\to
  R\otimes_{R[\Z]}F_1$.  Therefore $\det A(1)\in R$ is a unit, that
  is, the numerator of $\det A(1)$ is odd.  Since $\det A(1)-\det
  A(-1)$ is of the form $2\cdot r$ for some $r\in R$, it follows that
  $\det A(-1)$ has odd numerator too, that is, $\det A(-1)$ is a unit
  in~$\Z_{(2)}$.  Since $A(-1)$ represents
  $\bar\alpha:=1_{\Z^t_{(2)}}\otimes \alpha\colon
  \Z^t_{(2)}\otimes_{R[\Z]} F_0 \to \Z^t_{(2)}\otimes_{R[\Z]} F_1$, it
  follows that $\bar\alpha$ is an isomorphism.

  It is easily seen that the homomorphism $\Z_{(2)}^t \to \Z_{(2)}^t
  \otimes_{R[\Z]} \Z_{(2)}^t$ defined by $a \to 1\otimes a$ is an
  isomorphism, with inverse $a\otimes b \to ab$.  Now, the composition
  \[
  g\colon F_1 \to \Z_{(2)}^t \otimes_{R[\Z]} F_1
  \xrightarrow{\bar\alpha^{-1}} \Z_{(2)}^t
  \otimes_{R[\Z]} F_0 \xrightarrow{1_{\Z_{(2)}^t}\otimes f} \Z_{(2)}^t
  \otimes_{R[\Z]} \Z_{(2)}^t \cong
 \Z_{(2)}^t
  \]
  is a unique homomorphism such that $f=g\alpha$.
\end{proof}

\begin{proof}
  [Proof of Proposition~\ref{proposition:hat-Gamma-is-local}] 

  Recall that $\widehat\Gamma$ is given as the extension
  \[
  1 \to \Z_{(2)}/\Z \to \widehat\Gamma \to (\Z_{(2)}^t)^2 \rtimes \Z
  \to 1.
  \]
  By Lemma~\ref{lemma:Z^t_(2)-is-local} and
  Theorem~\ref{theorem:extension-of-local-group-by-local-module},
  $G:=(\Z_{(2)}^t)^2 \rtimes \Z$ is $R$-local.  Since the above
  extension is central, $\Z_{(2)}/\Z$ is an $RG$-local module by
  Lemma~\ref{lemma:trivial-G-module-is-RG-local}.  By
  Theorem~\ref{theorem:extension-of-local-group-by-local-module}, it
  follows that $\widehat\Gamma$ is $R$-local.
\end{proof}

\bibliographystyle{amsalpha} \renewcommand{\MR}[1]{}

\bibliography{research}

\end{document}